\documentclass[12pt]{amsart}
\usepackage{amsfonts,latexsym,amsmath, amssymb}
\usepackage{mathrsfs,MnSymbol}
\usepackage{url,color}
\usepackage{upgreek}
\usepackage{fancyhdr}
\usepackage{hyperref}
\usepackage{times}
\usepackage{scalefnt}
\usepackage{url}
\usepackage{cite}

\newcommand{\bea}{\begin{eqnarray}}
\newcommand{\eea}{\end{eqnarray}}
\def\beaa{\begin{eqnarray*}}
\def\eeaa{\end{eqnarray*}}
\def\ba{\begin{array}}
\def\ea{\end{array}}
\def\be#1{\begin{equation} \label{#1}}
\def \eeq{\end{equation}}

\def\be{{\beta}}

\newtheorem{theorem}{Theorem}[section]
\newtheorem{lemma}[theorem]{Lemma}
\newtheorem{proposition}[theorem]{Proposition}
\newtheorem{corollary}[theorem]{Corollary}
\newtheorem{definition}[theorem]{Definition}
\newtheorem{remark}[theorem]{Remark}

\numberwithin{equation}{section}

\numberwithin{equation}{section}

\topmargin - .23 in

\begin{document}


\title{On scattering for the cubic defocusing nonlinear Schr\"odinger equation on waveguide $\mathbb{R}^2\times \mathbb{T}$  }
\author{Xing Cheng$^{*}$, Zihua Guo$^{**}$,  Kailong Yang$^{\dagger}$ \and Lifeng Zhao$^{\dagger}$}

\thanks{$^*$ College of Science, Hohai University, Nanjing 210098,\ Jiangsu,\   China. \texttt{chengx@hhu.edu.cn}}

\thanks{$^{**}$ School of Mathematical Sciences, Monash University, \ VIC 3800,\  Australia， \texttt{zihua.guo@monash.edu}  }

\thanks{$^{\dagger}$ Wu Wen-Tsun Key Laboratory of Mathematics, Chinese Academy of Sciences
 and Department of Mathematics, University of Science and Technology of China, Hefei 230026, \ Anhui, \ China. \texttt{ykailong@mail.ustc.edu.cn, zhaolf@ustc.edu.cn} }

\thanks{$^{*}$ Xing Cheng has been partially supported by the
NSF grant of China (No. 11526072, No. 51509073).
}	

\thanks{$^{\dagger}$ Lifeng Zhao has been partially supported by the
NSF grant of China (No. 10901148, No. 11371337) and is also supported by Youth Innovation Promotion Association CAS}	
\begin{abstract}
	In this article, we will show the global wellposedness and scattering of the cubic defocusing nonlinear Schr\"odinger equation on waveguide $\mathbb{R}^2\times \mathbb{T}$ in $H^1$. We first establish the linear profile decomposition in $H^{ 1}(\mathbb{R}^2 \times \mathbb{T})$ motivated by the linear profile decomposition of the mass-critical Schr\"odinger equation in $L^2(\mathbb{R}^2)$. Then by using the
solution of the infinite dimensional vector-valued resonant nonlinear Schr\"odinger system
to approximate the nonlinear profile, we can prove scattering in $H^1$ by using the concentration-compactness/rigidity method.
\bigskip

\end{abstract}

\maketitle

\setcounter{tocdepth}{2}
\pagenumbering{roman} \tableofcontents \newpage \pagenumbering{arabic}

\section{Introduction}
In this article, we will consider the cubic nonlinear Schr\"odinger equation on $\mathbb{R}^2\times \mathbb{T}$:
\begin{equation}\label{eq1.1}
\begin{cases}
i\partial_t u + \Delta_{\mathbb{R}^2\times \mathbb{T}}  u = |u|^2u,\\
u(0) = u_0\in H^1(\mathbb{R}^2\times \mathbb{T}),
\end{cases}
\end{equation}
where $\Delta_{\mathbb{R}_x^2  \times \mathbb{T}_y} $ is the Laplace-Beltrami operator on $\mathbb{R}^2 \times \mathbb{T}$ and $u : \mathbb{R} \times \mathbb{R}^2\times \mathbb{T}  \to \mathbb{C}$ is a complex-valued function.

 The equation \eqref{eq1.1} has the following conserved quantities:
\begin{align*}
\text{mass conservation: }    &\quad   \mathcal{M}(u(t))  = \int_{\mathbb{R}^2 \times \mathbb{T}} |u(t,x,y)|^2\,\mathrm{d}x\mathrm{d}y,\\
\text{   energy conservation:  }     &  \quad \mathcal{E}(u(t))  = \int_{\mathbb{R}^2 \times \mathbb{T}} \frac12 |\nabla u(t,x,y)|^2  + \frac14 |u(t,x,y)|^4 \,\mathrm{d}x\mathrm{d}y,\\
\text{ momentum conservation: } &  \quad   \mathcal{P}(u(t)) = \Im \int_{\mathbb{R}_x^2\times \mathbb{T}_y} \overline{u}(x,y,t) \nabla u(x,y,t)\,\mathrm{d}x\mathrm{d}y.
\end{align*}
The equation \eqref{eq1.1} is a special case of the general nonlinear Schr\"odinger equation on the waveguide $\mathbb{R}^n \times \mathbb{T}^m$:
\begin{equation}\label{eq1.2n}
\begin{cases}
i\partial_t u + \Delta_{\mathbb{R}^n \times \mathbb{T}^m}  u = |u|^{p-1} u,\\
u(0) = u_0\in H^1(\mathbb{R}^n\times \mathbb{T}^m),
\end{cases}
\end{equation}
where $1 < p < \infty$, $m,n \in \mathbb{Z}$, and $m, n\ge 1$. This kind of equations arise as models in the study of nonlinear optics (propagation of laser beams through the atmosphere or in a plasma), especially in nonlinear optics of telecommunications \cite{S,SL}.

We are interested in the range of $p$ for wellposedness and scattering of \eqref{eq1.2n} on $\mathbb{R}^n\times \mathbb{T}^m$.
On one hand, the wellposedness is intuitively determined by the local geometry of the manifold $\mathbb{R}^n \times \mathbb{T}^m$. Because the manifold is locally just $\mathbb{R}^n \times \mathbb{R}^m$, we believe the wellposedness is the same as the Euclidean case, that is when $1<  p \le 1 + \frac4{m+n-2}$ the wellposedness is expected. Just as the Euclidean case, we call the equation energy-subcritical when $ 1< p < 1 + \frac4{m+n-2}$, $m,n \ge 1$ and energy-critical when $p = 1 + \frac4{m+n-2}$, $m + n \ge 3$, $m,n\ge 1$. On the other hand, scattering is expected to be determined by the asymptotic volume growth of ball with radius $r$ in the manifold $\mathbb{R}^2\times \mathbb{T}$ when $r\to \infty$. From the heuristic that linear solutions with frequency $\sim N$ initially localized around the origin will disperse at time t in the ball of radius $\sim N t$, scattering is expected to be partly determined by the asymptotic volume growth of balls with respect to their radius. Since $\inf_{ z \in \mathbb{R}^n \times \mathbb{T}^m } \text{Vol}_{\mathbb{R}^n \times \mathbb{T}^m} (B(z,r)) \sim r^n, \text{ as } r \to \infty$, the linear solution is expected to decay at a rate $\sim t^{-\frac{n}2}$ and based on the scattering theory on $\mathbb{R}^n$,
the solution of \eqref{eq1.2n} is expected to scatter for $ p \ge 1 + \frac4n$. Moreover,
modified scattering in the small data case is expected for $ 1 + \frac2n < p < 1 + \frac4n$ when $n \ge 2$ or $ 2 < p < 5$ when $n = 1$.
 Therefore, regarding heuristic on the wellposedness and scattering, we expect the solution of \eqref{eq1.2n} globally exists and scatters in the range $ 1 + \frac4n \le p \le 1 + \frac4{m+n-2}$. For $ 1 + \frac2n < p < 1 + \frac4n$ when $n \ge 2$ or $ 2 < p < 5$ when $n = 1$, modified scattering is expected as in the Euclidean space case for small data.

The nonlinear Schr\"odinger equations on the waveguide have been intensively studied in the last decades. When $ n= m = 1$, H. Takaoka and N. Tzvetkov \cite{TT} proved global wellposedness for
any data in $L^2$ when $1< p < 3$ and global wellposedness for small data in $L^2$ when $p = 3$. This work heavily relies on the techniques developed by J. Bourgain \cite{B}, where the $X^{s,b}$ space is used to show the wellposedness when $1< p < 3$, and the Hardy-Littlewood circle method is used to establish the Strichartz estimate when $p = 3$ as in \cite{B}. Later, S. Herr, D. Tataru and N. Tzvetkov \cite{HTT1} considered the cubic nonlinear Schr\"odinger equation on
$\mathbb{R}^n \times \mathbb{T}^{m}$ with $n + m = 4$ and $0 \le n \le 3$. In particular,
they proved global wellposedness for small data in $H^s$, $s\ge 1$ in the case $\mathbb{R}^2\times \mathbb{T}^2$,
$\mathbb{R}^3\times \mathbb{T}$ after establishing the corresponding Strichartz estimate, where the trilinear estimates in the context of the $U^p-$ and $V^p-$type spaces is used as in \cite{HTT} to deal with the nonlinear term in the critical space. Recently, Z. Hani, B. Pausader, N. Tzvetkov and N. Visciglia \cite{HPTV} considered the cubic
nonlinear Schr\"odinger equation posed on the spatial domain $\mathbb{R}\times \mathbb{T}^m$, where $1\le m \le 4$.
They proved modified scattering and constructed modified wave operators for small initial and final data.

 In \cite{IP1}, A. D. Ionescu and B. Pausader proved global wellposedness in $H^1$ for the cubic defocusing nonlinear Schr\"odinger equation
 on $\mathbb{R}\times \mathbb{T}^3$.
In the article, they first establish a scale invariant Strichartz estimate
$\|e^{it\Delta_{\mathbb{R}\times \mathbb{T}^3}}P_N f\|_{L^q([-1,1]\times \mathbb{R}\times \mathbb{T}^3)} \lesssim N^{2-\frac6q}\|f\|_{L^2}$, where $q > \frac{18}5$, $N\ge 1$. Then a linear profile decomposition consists of $ e^{-it_k \Delta_{\mathbb{R}\times \mathbb{T}^3}} \left(\lambda_k^{-1}\phi_k(\lambda_k^{-1} \Psi (x-x_k))\right)$, where $\lambda_k\to 0$, $x_k\in \mathbb{R}\times \mathbb{T}^3$, $\phi_k\in \dot{H}^1(\mathbb{R}^4)$, $\Psi $ is a local diffeomorphism from
$\mathbb{R}\times \mathbb{T}^3$ to $\mathbb{R}^4$ and $\tilde{\phi}_k(x-x_k)$, where $\tilde{\phi}_k \in H^1(\mathbb{R}\times \mathbb{T}^3)$, was established similar to \cite{IP,IPS}.
To show the space-time control of the nonlinear profiles associated with the profile $ e^{-it_k \Delta_{\mathbb{R}\times \mathbb{T}^3}} \left(\lambda_k^{-1}\phi_k(\lambda_k^{-1} \Psi (x-x_k))\right)$ in the linear profile decomposition,
the scattering of the energy-critical nonlinear Schr\"odinger equation on $\mathbb{R}^4$ is used.
 N. Tzvetkov and N. Visciglia \cite{TV2} studied the Cauchy problem and large
data scattering for the energy subcritical nonlinear Schr\"odinger equation on $\mathbb{R}^n \times \mathbb{T}$ in $H^1$,
where $n \ge 1$ and $1+  \frac4n < p < 1 + \frac4{n-1}$. In the article,
 by the Strichartz estimate, local wellposedness has been established. After proving the Morawetz estimate, global wellposedness and scattering
have been established.
When $n = 1$ and $m = 2$, Z. Hani and B. Pausader \cite{HP} consider the quintic nonlinear Schr\"odinger equation. They show global wellposedness and large-data scattering in $H^1(\mathbb{R} \times \mathbb{T}^2)$ in the defocusing case. In the article, to obtain the local existence, a local in time Strichartz estimate
$\|e^{it\Delta_{\mathbb{R}\times \mathbb{T}^2}} P_N f\|_{L_{t,x,y}^q([-1,1]\times \mathbb{R}\times \mathbb{T}^2)} \lesssim
N^{\frac32-\frac5q} \|f\|_{L_{x,y}^2(\mathbb{R}\times \mathbb{T}^2)}$, where $N \ge 1$ and $q > 4$
is sufficient. However, to obtain the asymptotic behavior, they need some global in time Strichartz estimate. By using the Strichartz estimate of the $\mathbb{R}$ direction of the whole manifold $\mathbb{R} \times \mathbb{T}^2$ and the Sobolev embedding, there is a natural global in time Strichartz estimate
$\|e^{it\Delta_{\mathbb{R}\times \mathbb{T}^2}} f\|_{L_{t,x,y}^6(\mathbb{R}\times \mathbb{R} \times \mathbb{T}^2)} \lesssim \|f\|_{H^\frac23_{x,y}}$,
this loss of $\frac23$ derivatives does not allow the local wellposedness theory in $H^1$.
Thus, to overcome the difficulty, a global in time integrability and
less derivative loss type Strichartz estimate is needed. By $TT^*$ argument, decomposing the relevant inner product into a diagonal part and a non-diagonal part. The diagonal component leads to the loss of derivatives but gives a contribution that has $l_\gamma^2$ time integrability, while the nondiagonal part loses fewer derivatives but brings slower $l_\gamma^q$($4< q < \infty$) time integrability. They develop a global Strichartz estimate of the form
\begin{align*}
\|e^{it\Delta_{\mathbb{R}\times \mathbb{T}^2}} f\|_{l_\gamma^q L_{x,y,t}^r (\mathbb{R} \times \mathbb{T}^2 \times [\gamma, \gamma + 1])}
\lesssim \|\langle \nabla \rangle^{\frac32 - \frac5r} f\|_{L_{x,y}^2(\mathbb{R}\times \mathbb{T}^2)},
\end{align*}
where $\frac2q + \frac1r = \frac12$, $4 < q,r < \infty$.
To give the existence of the critical element, they establish a linear profile decomposition, where another kind of profile looks like $e^{-it_k\Delta_{\mathbb{R}
\times \mathbb{T}^2}}(e^{i(x-x_k)\xi_k} \lambda_k^{-\frac12} \phi_k(\lambda_k^{-1}(x-x_k,y))) $,
where $\lambda_k\to \infty$, $\xi_k\in \mathbb{R}$, $x_k\in \mathbb{R}$, $\phi_k\in H^{0,1}$ was included in the linear
profile decomposition. To exclude the nonlinear profile associated with the large-scale profile in the linear profile decomposition,
they assume the scattering of the quintic resonant system
\begin{equation*}
i\partial_t u_j + \Delta_{\mathbb{R}} u_j = \sum\limits_{\substack{ ( j_1,j_2,j_3,j_4,j_5)}\in \mathcal{Q}(j)} u_{j_1} \bar{u}_{j_2} u_{j_3} \bar{u}_{j_4} u_{j_5}, \  j\in \mathbb{Z}^2,
\end{equation*}
where
$\mathcal{Q}(j) = \{(j_1,j_2,j_3,j_4,j_5)\in (\mathbb{Z}^2)^5: j_1-j_2 + j_3 - j_4 + j_5 = j, |j_1|^2 - |j_2|^2 + |j_3|^2 - |j_4|^2 + |j_5|^2 = |j|^2\}$.
We also refer to \cite{Ta,TV} on the wellposedness and scattering of the nonlinear Schr\"odinger equation on general waveguide $\mathbb{R}^n\times \mathbb{M}^m$, where $\mathbb{M}^m$
is a compact $m-$dimensional Riemann manifold.

We now give a summary of the known results of NLS on $\mathbb{R}^n\times \mathbb{T}^m$ for $n,m \in \mathbb{Z}$, and $n,m \ge 1 $ in the following table:
\textcolor{red}{
\begin{table}[h]
\begin{center}
\begin{tabular}{|c|c|c|c|c|}
\hline    $ \mathbb{R}^n\times \mathbb{T}^m$                & Results                               \\
\hline   $ \mathbb{R}\times \mathbb{T}$                   & GWP in $L^2$($1< p < 3$), GWP in $L^2$, modified scattering  ($ p  = 3$, small data)     \\
\hline  $\mathbb{R}\times \mathbb{T}^2$               & LWP \& GWP in $H^1$,  modified scattering ($ p = 3$, small data),  Scattering in $H^1$($p  = 5$)    \\
\hline   $\mathbb{R}\times \mathbb{T}^3$              & LWP  \& GWP in $H^1$,  modified scattering ($ p  = 3$, small data)    \\
\hline   $ \mathbb{R}^2 \times \mathbb{T} $        & GWP \& Scattering in $H^1$($3<p  <5$)                   \\
\hline   $ \mathbb{R}^2\times \mathbb{T}^2$           & LWP \& GWP in $H^1$($p  = 3$, small data)                        \\
\hline   $ \mathbb{R}^3 \times \mathbb{T} $        & GWP \& Scattering in $H^1$($\frac73 <p  <3$)          \\
 \hline
\end{tabular}
\end{center}
\caption{We only state the case for $n+m \le 4$.
 }\label{Table1}
\end{table}
}

Our main result addresses the scattering for \eqref{eq1.1} in $H^1(\mathbb{R}^2 \times \mathbb{T})$:
\begin{theorem}[Scattering in $H^1(\mathbb{R}^2 \times \mathbb{T})$] \label{th1.3}
For any initial data $u_0\in H^1(\mathbb{R}^2 \times \mathbb{T})$, there exists a solution $u\in C_t^0 H_{x,y}^1(\mathbb{R}\times \mathbb{R}^2 \times \mathbb{T})$ that is global and scatters in the sense that there exist $u^\pm \in H^1(\mathbb{R}^2\times \mathbb{T})$ such that
\begin{equation*}
\|u(t)- e^{it\Delta_{\mathbb{R}^2\times \mathbb{T}}} u^\pm \|_{H^1(\mathbb{R}^2\times \mathbb{T})} \to 0,  \text{  as }  t\to \pm \infty.
\end{equation*}
\end{theorem}
 The proof of Theorem \ref{th1.3} is based on the concentration compactness/rigidity method developed by C. E. Kenig and F. Merle \cite{KM}.
For \eqref{eq1.1}, the dispersive effect of the $\mathbb{R}^2$-component is strong enough, to give a
 global Strichartz estimate \cite{TV,TV2}:
\begin{align}\label{eq1.321}
\|e^{it\Delta_{\mathbb{R}^2 \times \mathbb{T}}} f\|_{L_{t,x}^4 H_y^1 \cap L_t^4 W_x^{1,4} L_y^2(\mathbb{R}\times \mathbb{R}^2 \times \mathbb{T})} \lesssim  \|f\|_{H_{x,y}^{1}},
\end{align}
where $(q,r)$ is $L^2$-admissible on $\mathbb{R}^2$.
By the wellposedness and scattering theory, we observe to prove Theorem \ref{th1.3}, we only need to prove the solution satisfies a weaker
space norm $L_{t,x}^4 H_y^{1-\epsilon_0}$, where $0 < \epsilon_0 < \frac12$ is some fixed number used hereafter.
So we only need a linear profile decomposition in $H^1(\mathbb{R}^2 \times \mathbb{T})$ with the remainder in $L_{t,x}^4 H_y^{1-\epsilon_0}$,
which is essentially equivalent to describe the defect of compactness of
\begin{equation}\label{eq }
e^{it\Delta_{\mathbb{R}^2\times \mathbb{T}}}: H^{ 1}_{x, y} (\mathbb{R}^2\times \mathbb{T})\hookrightarrow L_{t,x}^4 H_y^{1-\epsilon_0}(\mathbb{R}\times \mathbb{R}^2 \times\mathbb{T}),
\end{equation}
we can then establish a linear profile decomposition similar to \cite{HP}. However, the argument is mainly based on the argument to establish the linear
profile decomposition of the Sch\"odinger equation
in $L^2(\mathbb{R}^2)$.

The nonlinear profiles can be defined to be the solution of the cubic nonlinear Schr\"odinger equation on $\mathbb{R}^2 \times \mathbb{T}$, with initial data is each profile in the linear profile decomposition. Just as in \cite{HP}, the nonlinear profile can be approximated by applying $e^{it\Delta_{\mathbb{T}}}$ to the rescaling of the solution of the cubic resonant Schr\"odinger system
\begin{equation}\label{eq1.733}
\begin{cases}
i\partial_t u_j + \Delta_{\mathbb{R}^2} u_j = \sum\limits_{j_1,j_2,j_3 \in \mathcal{R}(j) } u_{j_1} \bar{u}_{j_2} u_{j_3},\\
u_j(0) = u_{0,j},\ j\in \mathbb{Z},
\end{cases}
\end{equation}
 where $\mathcal{R}(j) =
 \{ j_1,j_2,j_3 \in \mathbb{Z}: j_1-j_2+j_3= j, \, |j_1|^2 - |j_2|^2 + |j_3|^2 = |j|^2\}$ in the large scale case.
To derive the almost-periodic solution, we also need the following scattering theorem of the cubic resonant Schr\"odinger system, which is proved by using the
argument in \cite{D2} to deal with the cubic nonlinear Schr\"odinger equation in $L^2(\mathbb{R}^2)$:
\begin{theorem}[Scattering of the cubic resonant Schr\"odinger system, \cite{YZ}]\label{th1.2}
Let $E>0$, for any initial data $\vec{u}_0$ satisfying
\begin{equation*}
\|\vec{u}_0\|_{L_x^2 h^1} :=   \left\|\Big(\sum\limits_{j\in \mathbb{Z}} \langle j\rangle^2 |u_{0,j}|^2\Big)^\frac12 \right\|_{L^2(\mathbb{R}^2)} \le E,
\end{equation*}
there exists a global solution to \eqref{eq1.733},
where $\vec{u} = \{u_j\}_{j\in \mathbb{Z}},$
with $\|\vec{u}(t)\|_{L_x^2 h^1} = \|\vec{u}_0\|_{L_x^2 h^1}$ satisfying
\begin{equation}\label{eq5.19new}
\|\vec{u}\|_{L_{t,x}^4 h^1(\mathbb{R}\times \mathbb{R}^2 \times \mathbb{Z})}  := \bigg\| \Big( \sum\limits_{j\in \mathbb{Z}} \langle j\rangle^2 |u_j|^2 \Big)^\frac12   \bigg\|_{L_{t,x}^4(\mathbb{R}\times \mathbb{R}^2)}  \le
C,
\end{equation}
for some constant $C $ depends only on $\|\vec{u}_0\|_{L_x^2 h^1}$.
In addition, the solution scatters in $L_x^2 h^1$ in the sense that there exists $\{u_j^{\pm }\}_j \in  L^2_x h^1$
such that
\begin{equation}
  \bigg \| \Big( \sum\limits_{j\in \mathbb{Z}} \langle j\rangle^2 |  u_j(t) - e^{it\Delta_{\mathbb{R}^2}} u_j^{\pm  }|^2 \Big)^\frac12  \bigg\|_{L^2(\mathbb{R}^2)} \to 0, \text{ as } t\to \pm \infty.
\end{equation}
\end{theorem}
By using the concentration-contradiction, the existence of an almost-periodic solution is given
in $H^1(\mathbb{R}^2 \times \mathbb{T})$.
By the interaction Morawetz quantity, the critical element can be killed.
\begin{remark}
The argument in the proof of Theorem \ref{th1.3} in fact does not rely on the structure of $\mathbb{T}$ in the manifold $\mathbb{R}^2\times \mathbb{T}$.
Thus, Theorem \ref{th1.3} can be generalized to cubic nonlinear Schr\"odinger equation on $\mathbb{R}^2\times \mathbb{M}$,
where $\mathbb{M}$ is a one dimensional compact Riemann manifold as in \cite{TV}.
\end{remark}
The rest of the paper is organized as follows. After introducing some notations and
preliminaries, we give the local wellposedness and small data scattering in Section \ref{se2}. We also give the stability theory in this section.
In Section \ref{se3}, we derive the linear profile decomposition for data in $H^{ 1}(\mathbb{R}^2 \times \mathbb{T})$ and analyze the nonlinear profiles.
In Section \ref{se4}, we reduce the non-scattering in $H^1$ to the existence of an almost-periodic solution and show the extinction
of such an almost-periodic solution in Section \ref{se5}.
\subsection{Notation and Preliminaries}
We will use the notation $X\lesssim Y$ whenever there exists some constant $C>0$ so that $X \le C Y$. Similarly, we will use $X \sim Y$ if
$X\lesssim Y \lesssim X$.

We define the torus to be $\mathbb{T} = \mathbb{R}/(2\pi \mathbb{Z})$.

In the following, we will frequently use some space-time norm, we now give the definition of it.

For any $I \subset \mathbb{R}$, $u(t,x,y): I \times \mathbb{R}^2 \times \mathbb{T} \to \mathbb{C}$, define the space-time norm
\begin{align*}
\|u\|_{L_t^p L_x^q L_y^2(I\times \mathbb{R}^2\times \mathbb{T})}  & = \left\|\left\|\Big( \int_{\mathbb{T}} |u(t,x,y)|^2 \,\mathrm{d}y \Big)^\frac12\right\|_{L_x^q(\mathbb{R}^2)}\right\|_{L_t^p(I)},\\
\|u\|_{H^1_{x,y}}  & = \|\langle \nabla_x \rangle u\|_{L_{x,y}^2} + \|\langle \nabla_y \rangle u\|_{L_{x,y}^2}.
\end{align*}
We will frequently use the partial Fourier transform and partial space-time Fourier transform: For $f(x,y): \mathbb{R}^2 \times \mathbb{T} \to \mathbb{C}$,
\begin{align*}
\mathcal{F}_x f(\xi,y) = \frac1{2\pi} \int_{\mathbb{R}^2} e^{-ix\xi} f(x,y) \,\mathrm{d}x.
\end{align*}
Given $H: \mathbb{R} \times \mathbb{R}^2 \times \mathbb{T}\to \mathbb{C}$, we denote the partial space-time Fourier transform to be
\begin{align*}
\mathcal{F}_{t,x}{H}(\omega,\xi,y) = \frac1{(2\pi)^\frac32} \int_{\mathbb{R}} \int_{\mathbb{R}^2} e^{i\omega t -i\xi x} H(t,x,y) \,\mathrm{d}x\mathrm{d}t.
\end{align*}
We also define the partial Littlewood-Paley projectors $P_{\le N}^x$ and $P_{\ge N}^x$ as follows:
fix a real-valued radially symmetric bump function $\varphi(\xi)$ satisfying
\begin{equation*}
\varphi(\xi) =
\begin{cases}
1, \ |\xi|\le 1,\\
0, \ |\xi|\ge 2,
\end{cases}
\end{equation*}
for any dyadic number $N\in 2^{\mathbb{Z}}$, let
\begin{align*}
\mathcal{F}_x (P_{\le N}^x f)(\xi,y) = \varphi \left(\frac\xi N\right) (\mathcal{F}_x f)(\xi,y),\\
\mathcal{F}_x (P_{\le N}^x f)(\xi,y) = \left(1-\varphi \left(\frac\xi N\right) \right)(\mathcal{F}_x f)(\xi,y)
\end{align*}
We now define the discrete nonisotropic Sobolev space. For $\vec{\phi} = \{\phi_k\}_{k\in \mathbb{Z}}$ a sequence of real-variable functions, we define
\begin{align*}
 H^{s_1}_x h^{s_2} = \left\{ \vec{\phi} = \{ \phi_k\}: \|\vec{\phi}\|_{H^{s_1} h^{s_2}}  =   \bigg \| \Big( \sum\limits_{k\in \mathbb{Z}} \langle k\rangle^{2s_2} |\phi_k (x) |^2  \Big)^\frac12 \bigg\|_{H_x^{s_1}}  < \infty \right \},
\end{align*}
where $s_1,s_2\ge 0$. In particular, when $s_1 = 0$, we denote the space $H_x^{s_1} h^{s_2}$ to be $L_x^2 h^{s_2}$.
For $\psi\in L_x^2 H_y^{1}(\mathbb{R}^2\times \mathbb{T})$, we have the vector $\vec{\psi}= \{\psi_k\} \in L_x^2 h^1 $, where $\psi_k$ is the sequence of periodic Fourier coefficients of $\psi$ defined by
\begin{align*}
\psi_k(x) = \frac1{(2\pi)^\frac12} \int_{\mathbb{T}} \psi(x,y) e^{-iky} \,\mathrm{d}y.
\end{align*}
Throughout the article, $0< \epsilon_0 < \frac12$ is some fixed number.

\section{Local wellposedness and small data scattering}\label{se2}
 In this section, we will review the local wellposedness and small data scattering, that is Theorem \ref{th2.3} and Theorem \ref{th2.4}. These results have been established in \cite{TV,TV2}. We also give the stability theory which will be used in showing the existence of a critical element in Section \ref{se4}.

We first recall the following Strichartz estimate, which is established in \cite{TV}.
\begin{proposition}[Strichartz estimate]
\begin{align}
\left\|e^{it\Delta_{\mathbb{R}^2 \times \mathbb{T}}} f\right\|_{L_t^p L_x^q L_y^2} \lesssim \|f\|_{L_{x,y}^2(\mathbb{R}_x^2\times \mathbb{T}_y)}, \label{eq2.1} \\
\left\|\int_0^t e^{i(t-s)\Delta_{\mathbb{R}^2 \times \mathbb{T}}} F(s,x,y)\,\mathrm{d}s\right\|_{L_t^p L_x^q L_y^2} \lesssim \|F\|_{L_t^{\tilde{p}'} L_x^{\tilde{q}'} L_y^2}, \label{eq2.1'}
\end{align}
where $(p,q),\, (\tilde{p},\tilde{q})$ satisfies $\frac2p + \frac2q = 1$, $\frac2{\tilde{p}} + \frac2{\tilde{q}} = 1$, and $2 < p, \tilde{p} \le \infty$.
\end{proposition}
The following nonlinear estimate is useful in showing the local wellposedness.
\begin{proposition}[Nonlinear estimate] \label{pr3.3}
\begin{align}\label{eq2.8new}
\| u_1 u_2 u_3\|_{L_t^\frac43 L_x^\frac43 H_y^{1-\epsilon_0}}   \lesssim  \| u_1\|_{L_t^4 L_x^4 H_y^{1-\epsilon_0}}  \|  u_2\|_{L_t^4 L_x^4 H_y^{1-\epsilon_0}} \|  u_3\|_{L_t^4 L_x^4 H_y^{1-\epsilon_0}}.
\end{align}
\end{proposition}
\begin{proof}
Since $H^{1-\epsilon_0}_y(\mathbb{T})$ is an algebra, we have
\begin{equation*}
\| u_1 u_2 u_3\|_{H_y^{1-\epsilon_0}} \lesssim \|  u_1 \|_{H_y^{1-\epsilon_0}} \|  u_2\|_{H_y^{1-\epsilon_0}} \|u_3\|_{H_y^{1-\epsilon_0}}.
\end{equation*}
By the H\"older inequality,
\begin{align*}
   \big\| \|  u_1 \|_{H_y^{1-\epsilon_0}}  \|  u_2  \|_{H_y^{1-\epsilon_0}}   \|  u_3 \|_{H_y^{1-\epsilon_0}}  \big\|_{L_t^\frac43 L_x^\frac43}
\lesssim   \  \|  u_1\|_{L_t^4 L_x^4 H_y^{1-\epsilon_0}}  \|  u_2\|_{L_t^4 L_x^4 H_y^{1-\epsilon_0}} \|  u_3\|_{L_t^4 L_x^4 H_y^{1-\epsilon_0}},
\end{align*}
so we have \eqref{eq2.8new}.
\end{proof}
By the Strichartz estimate and the nonlinear estimate, we can give the local wellposedness and small data scattering in $L_x^2 H_y^1(\mathbb{R}^2\times \mathbb{T}  )$ and $H_{x,y}^{1}(\mathbb{R}^2\times \mathbb{T})$ easily.
\begin{theorem}[Local wellposedness]\label{th2.3}
For any $E>0$, suppose that $\|u_0\|_{L_x^2 H_{y}^{1 }(\mathbb{R}^2 \times \mathbb{T})} \le E$, there exists $\delta_0 = \delta_0(E)>0$ such that if
\begin{equation*}
\|e^{it\Delta_{\mathbb{R}^2\times \mathbb{T}}} u_0 \|_{L_t^4 L_x^4 H_y^{1-\epsilon_0}  (I\times \mathbb{R}^2\times \mathbb{T})} \le   \delta_0,
\end{equation*}
where $I$ is a time interval, then there exits a unique solution $u\in C_t^0  L_x^2 H_{ y}^{ 1 }(I\times \mathbb{R}^2 \times \mathbb{T})$ of \eqref{eq1.1}
satisfying
\begin{align*}
\|u\|_{L_t^4 L_x^4 H_y^{1-\epsilon_0} } & \le 2 \|e^{it\Delta_{ \mathbb{R}^2\times \mathbb{T} }} u_0\|_{L_t^4 L_x^4 H_y^{1-\epsilon_0}   },\\
\|u\|_{L_t^\infty L_x^2 H_{ y}^{ 1 }  }  & \le C\|u_0\|_{ L_x^2 H_{ y}^{ 1 }  }.
\end{align*}
Moreover, if $u_0\in H_{x,y}^1(\mathbb{R}^2 \times \mathbb{T})$, then $u \in C_t^0 H_{x,y}^1(I \times \mathbb{R}^2 \times \mathbb{T})$ with
\begin{align*}
\|u\|_{L_t^\infty H_{x,y}^1} \le C(E)\|u_0\|_{H_{x,y}^1},\  \|u\|_{L_t^4 L_x^4 H_y^1 \cap L_t^4 W_x^{1,4} L_y^2} \le C(\|u_0\|_{H^1}).
\end{align*}
\end{theorem}
 Arguing as in \cite{C,T2}, we can easily obtain the global wellposedness by the Strichartz estimate together with the conservation of mass and energy:
\begin{theorem}[Global wellposedness in $H^1$] \label{th2.433}
 For any $E>0$, if $\|u_0\|_{H_{x,y}^1 (\mathbb{R}^2 \times \mathbb{T})} \le E$, there exists a unique global solution
$u\in C_t^0H_{x,y}^{1} \cap L_{t,x}^4 H_y^1 \cap L_t^4 W_x^{1,4} L_y^2(\mathbb{R} \times \mathbb{R}^2 \times \mathbb{T})$ of \eqref{eq1.1}
satisfying
\begin{align*}
\|u\|_{L_t^4 L_x^4 H_y^{1 } \cap L_t^4 W_x^{1,4} L_y^2} & \le C(E) \|u_0\|_{H^1},\\
\|u\|_{L_t^\infty H_{x,y}^{ 1 }  }  & \le C(E) \|u_0\|_{H_{x,y}^{ 1 } }.
\end{align*}
\end{theorem}
\begin{theorem}[Small data scattering in $L_x^2 H_y^1$]\label{th2.4}
There exists $\delta >0$ such that if $u_0\in H_{x,y}^1$ and $\|u_0\|_{L_x^2 H_{ y}^{1}(\mathbb{R}_x^2 \times \mathbb{T}_y)} \le \delta$,
\eqref{eq1.1} has an unique global solution
$u(t,x,y) \in C_t^0 L_x^2 H_{ y}^{ 1} \cap L_t^4 L_x^4 H_y^1  $ and there exist $u_\pm \in L_x^2 H_{ y}^{ 1}(\mathbb{R}^2 \times \mathbb{T})$ such that
\begin{equation}
\|u(t,x,y)- e^{it\Delta_{\mathbb{R}^2 \times \mathbb{T}}} u_\pm(x,y)\|_{L_x^2 H_{ y}^{ 1}} \to 0, \ \text{ as } t\to \pm \infty.
\end{equation}
\end{theorem}
We now give the stability theory in $ L_x^2 H_y^{1-\epsilon_0 }(\mathbb{R}^2 \times \mathbb{T})$.
\begin{theorem}[Stability theory]\label{le2.6}
Let $I$ be a compact interval and let $\tilde{u}$ be an approximate solution to $i\partial_t u + \Delta_{\mathbb{R}^2 \times \mathbb{T}} u =|u|^2 u$ in the
sense that
$i\partial_t \tilde{u} + \Delta_{\mathbb{R}^2\times \mathbb{T}} \tilde{u} = |\tilde{u}|^2 \tilde{u} + e$
for some function $e$. Assume that
\begin{align*}
\|\tilde{u}\|_{ L_t^\infty L_x^2 H_y^{1-\epsilon_0 } } \le M, \ \quad
\|\tilde{u}\|_{L_t^4 L_x^4 H_y^{1-\epsilon_0} } \le L
\end{align*}
for some positive constants $M$ and $L$. Let $t_0\in I$ and let $u(t_0)$ obey
\begin{equation}\label{eq3.32}
\|u(t_0)- \tilde{u}(t_0)\|_{ L_x^2 H_y^{1-\epsilon_0 } } \le M'
\end{equation}
for some $M'> 0$. Moreover, assume the smallness conditions
\begin{align}
\|e^{i(t-t_0)\Delta_{\mathbb{R}^2\times \mathbb{T}}} (u(t_0)-\tilde{u}(t_0))\|_{L_t^4 L_x^4 H_y^{1-\epsilon_0}  } \le \epsilon, \label{eq3.33}\\
\|e\|_{L_t^\frac43 L_x^\frac43 H_y^{1-\epsilon_0}  } \le \epsilon, \label{eq3.34}
\end{align}
for some $0 < \epsilon \le \epsilon_1$, where $\epsilon_1 = \epsilon_1(M,M',L) > 0$ is a small constant. Then, there exists a solution $u$ to $i\partial_t u + \Delta_{\mathbb{R}^2\times \mathbb{T}} u = |u|^2 u$ on $I\times \mathbb{R}^2 \times \mathbb{T}$ with
initial data $u(t_0)$ at time $t=t_0$ satisfying
\begin{align*}
\|u-\tilde{u}\|_{L_t^4 L_x^4 H_y^{1-\epsilon_0}} & \le C(M,M',L)\epsilon,\\
\|u-\tilde{u}\|_{L_t^\infty L_x^2 H_{y}^{ 1-\epsilon_0}}  & \le C(M,M',L)M',\\
\|u\|_{L_t^\infty L_x^2 H_{y}^{ 1-\epsilon_0 }}  & \le C(M,M',L),\\
\|u \|_{L_t^4 L_x^4 H_y^{1-\epsilon_0}} & \le C(M,M',L) .
\end{align*}
\end{theorem}
We need a short-time perturbation to prove this theorem in $ L_x^2 H_y^{1-\epsilon_0}(\mathbb{R}^2 \times \mathbb{T})$.
\begin{lemma}[Short-time perturbation]\label{le3.6}
Let $I$ be a compact interval and let $\tilde{u}$ be an approximate solution to \eqref{eq1.1} in the sense that
$i\partial_t \tilde{u} + \Delta_{\mathbb{R}^2 \times \mathbb{T}} \tilde{u} = |\tilde{u}|^2 \tilde{u} + e$
for some function $e$. Assume that
\begin{equation}\label{eq2.13new}
\|\tilde{u}\|_{  L_t^\infty L_x^2 H_y^{1-\epsilon_0}(I\times \mathbb{R}^2 \times \mathbb{T})} \le M
\end{equation}
for some positive constant $M$. Let $t_0\in I$ and $u(t_0)$ be such that
\begin{equation}\label{eq3.20}
\|u(t_0)- \tilde{u}(t_0)\|_{    L_x^2 H_y^{1-\epsilon_0}} \le M'
\end{equation}
for some $M' >0$. Assume also the smallness conditions
\begin{align}
 \|\tilde{u}\|_{L_t^4 L_x^4 H_y^{1-\epsilon_0}  (I\times \mathbb{R}^2 \times \mathbb{T})}  & \le \epsilon, \label{eq3.21}\\
\|e^{i(t-t_0)\Delta} (u(t_0) - \tilde{u}(t_0))\|_{L_t^4 L_x^4 H_y^{1-\epsilon_0}  }  & \le \epsilon, \label{eq3.22}\\
\|e\|_{L_t^\frac43 L_x^\frac43 H_y^{1-\epsilon_0}  }  &  \le \epsilon, \label{eq3.23}
\end{align}
for some $0 < \epsilon \le \epsilon_1$, where $\epsilon_1 = \epsilon_1(M,M') > 0$ is a small constant. Then, there exists a solution $u$ to \eqref{eq1.1} on $I\times \mathbb{R}^2 \times \mathbb{T}$ with initial data $u(t_0)$ at time $t= t_0$ satisfying
\begin{align}
\|u-\tilde{u}\|_{L_t^4 L_x^4 H_y^{1-\epsilon_0} }  & \lesssim \epsilon, \label{eq3.24}\\
\|u-\tilde{u}\|_{  L_t^\infty L_x^2 H_y^{1-\epsilon_0}}   &  \lesssim M', \label{eq3.25}\\
\|u\|_{ L_t^\infty L_x^2 H_y^{1-\epsilon_0}}  &  \lesssim M + M',\label{eq3.26} \\
\big\||u|^2 u - |\tilde{u}|^2 \tilde{u}\big\|_{L_t^\frac43 L_x^\frac43 H_y^{1-\epsilon_0}  }  &  \lesssim \epsilon. \label{eq3.27}
\end{align}
\end{lemma}
\begin{proof}
By symmetry, we may assume $t_0 = \inf I$. Let $w =  u - \tilde{u}$, then $w$ satisfies the following
\begin{equation*}
\begin{cases}
i\partial_t w + \Delta_{\mathbb{R}^2 \times \mathbb{T}} w = |\tilde{u} + w|^2 (\tilde{u} + w) - |\tilde{u}|^2 \tilde{u} - e,\\
w(t_0) = u(t_0)- \tilde{u}(t_0).
\end{cases}
\end{equation*}
For $t\in I$, we define
\begin{equation*}
A(t) = \big\||\tilde{u} + w|^2(\tilde{u} + w) - |\tilde{u}|^2 \tilde{u}\big\|_{L_t^\frac43 L_x^\frac43 H_y^{1-\epsilon_0}  ([t_0,t]\times \mathbb{R}^2 \times \mathbb{T})}.
\end{equation*}
By \eqref{eq3.21},
\begin{align}
A(t) &  = \big\||\tilde{u} + w|^2(\tilde{u} + w) - |\tilde{u}|^2 \tilde{u}\big\|_{L_t^\frac43 L_x^\frac43 H_y^{1-\epsilon_0}  ([t_0,t]\times \mathbb{R}^2 \times \mathbb{T})}\notag\\
  &  \lesssim \|w\|_{L_t^4 L_x^4 H_y^{1-\epsilon_0}  }\left( \|\tilde{u}\|_{L_t^4 L_x^4 H_y^{1-\epsilon_0}}^2 + \|w\|_{L_t^4 L_x^4 H_y^{1-\epsilon_0} }^2\right)\notag \\
  &  \lesssim \|w\|_{L_t^4 L_x^4 H_y^{1-\epsilon_0}  ([t_0,t]\times \mathbb{R}^2 \times \mathbb{T})}^3 + \epsilon_1^2 \|w\|_{L_t^4 L_x^4 H_y^{1-\epsilon_0}  ([t_0,t]\times \mathbb{R}^2 \times \mathbb{T})}.  \label{eq3.28}
\end{align}
On the other hand, by Strichartz estimate, \eqref{eq3.22} and \eqref{eq3.23}, we get
\begin{align}
&\|w\|_{L_t^4 L_x^4 H_y^{1-\epsilon_0}  ([t_0,t]\times \mathbb{R}^2 \times \mathbb{T})} \notag \\
 \lesssim  & \|e^{i(t-t_0)\Delta} w(t_0)\|_{L_t^4 L_x^4 H_y^{1-\epsilon_0}  ([t_0,t]\times \mathbb{R}^2 \times \mathbb{T})}
+ A(t) + \|e\|_{L_t^\frac43 L_x^\frac43 H_y^{1-\epsilon_0}  ([t_0,t]\times \mathbb{R}^2 \times \mathbb{T})} \notag\\
\lesssim &   A(t) + \epsilon. \label{eq3.29}
\end{align}
Combining \eqref{eq3.28} and \eqref{eq3.29}, we obtain
\begin{equation*}
A(t) \lesssim (A(t) + \epsilon)^3 + \epsilon_1^2(A(t) + \epsilon).
\end{equation*}
A standard continuity argument then shows that if $\epsilon_1$ is taken sufficiently small,
\begin{align*}
A(t) \lesssim \epsilon, \ \forall\, t\in I,
\end{align*}
which implies \eqref{eq3.27}.

Using \eqref{eq3.27} and \eqref{eq3.29}, one easily derives \eqref{eq3.24}. Moreover, by Strichartz estimate, \eqref{eq3.20}, \eqref{eq3.23} and \eqref{eq3.27},
\begin{align*}
& \|w\|_{  L_t^\infty L_x^2 H_y^{1-\epsilon_0} (I\times \mathbb{R}^2 \times  \mathbb{T})}\\
  \lesssim  & \|w(t_0)\|_{   L_x^2 H_y^{1-\epsilon_0} } + \big\||\tilde{u} + w|^2(\tilde{u} + w) - |\tilde{u}|^2 \tilde{u}\big\|_{L_t^\frac43 L_x^\frac43 H_y^{1-\epsilon_0}  } + \|e\|_{L_t^\frac43 L_x^\frac43 H_y^{1-\epsilon_0}  }\\
\lesssim &  M' + \epsilon,
\end{align*}
which establishes \eqref{eq3.25} for $\epsilon_1= \epsilon_1(M')$ sufficiently small.

To prove \eqref{eq3.26}, we use Strichartz estimate, \eqref{eq2.13new}, \eqref{eq3.20}, \eqref{eq3.27} and \eqref{eq3.21},
\begin{align*}
  \|u\|_{  L_t^\infty L_x^2 H_y^{1-\epsilon_0}(I\times \mathbb{R}^2 \times  \mathbb{T})}
 \lesssim  & \|u(t_0)\|_{   L_x^2 H_y^{1-\epsilon_0}} + \big\||u|^2 u\big\|_{L_t^\frac43 L_x^\frac43 H_y^{1-\epsilon_0}  }\\
 \lesssim &  \|\tilde{u}(t_0) \|_{    L_x^2 H_y^{1-\epsilon_0}} +  \|u(t_0) - \tilde{u}(t_0) \|_{    L_x^2 H_y^{1-\epsilon_0}} \\
 & + \big\||u|^2 u - |\tilde{u}|^2 \tilde{u}\big\|_{L_t^\frac43 L_x^\frac43 H_y^{1-\epsilon_0}  } + \big\||\tilde{u}|^2 \tilde{u}\big\|_{L_t^\frac43 L_x^\frac43 H_y^{1-\epsilon_0}  }\\
 \lesssim &  M + M' + \epsilon+  \|\tilde{u}\|_{L_t^4 L_x^4 H_y^{1-\epsilon_0}  }^3
 \lesssim M + M'+ \epsilon + \epsilon_1^3.
\end{align*}
 Choosing $\epsilon_1 = \epsilon_1(M,M')$ sufficiently small, this finishes the proof of the lemma.

\end{proof}
We now turn to show the stability theory.
\begin{proof}
Subdivide $I$ into $J\sim ( 1 + \frac{L}{\epsilon_0})^4$ subintervals $I_j = [t_j, t_{j+1}]$, $0\le j\le J-1$ such that
\begin{equation*}
\|\tilde{u}\|_{L_t^4 L_x^4 H_y^{1-\epsilon_0}  (I_j\times \mathbb{R}^2\times\mathbb{T})} \le \epsilon_1,
\end{equation*}
where $\epsilon_1 = \epsilon_1(M, 2M')$ is as in Lemma \ref{le3.6}.

We need to replace $M'$ by $2M'$ as $\|u-\tilde{u}\|_{ L_t^\infty L_x^2 H_y^{1-\epsilon_0} }$ might grow slightly in time.

By choosing $\epsilon_1$ sufficiently small depending on $J$, $M$ and $M'$, we can apply Lemma \ref{le3.6} to obtain for each $j$ and all $0 < \epsilon < \epsilon_1$,
\begin{align*}
\|u-\tilde{u}\|_{L_t^4 L_x^4 H_y^{1-\epsilon_0}  (I_j\times \mathbb{R}^2\times\mathbb{T})}   &  \le C(j) \epsilon,\\
\|u-\tilde{u}\|_{ L_t^\infty L_x^2 H_y^{1-\epsilon_0} (I_j\times \mathbb{R}^2\times\mathbb{T})}   &  \le C(j) M',\\
\|u\|_{  L_t^\infty L_x^2 H_y^{1-\epsilon_0} (I_j\times \mathbb{R}^2\times\mathbb{T})}     &   \le C(j)(M + M'),\\
\big\||u|^2 u - |\tilde{u}|^2 \tilde{u}\big\|_{L_t^\frac43 L_x^\frac43 H_y^{1-\epsilon_0}  (I_j\times \mathbb{R}^2\times\mathbb{T})}   &  \le C(j) \epsilon,
\end{align*}
provided we can prove that analogues of \eqref{eq3.32} and \eqref{eq3.33} hold with $t_0$ replaced by $t_j$.

In order to verify this, we use an inductive argument. By Strichartz, \eqref{eq3.32}, \eqref{eq3.34} and the inductive hypothesis,
\begin{align*}
 & \|u(t_j)-\tilde{u}(t_j)\|_{  L_x^2 H_y^{1-\epsilon_0} }\\
\lesssim & \  \|u(t_0)- \tilde{u}(t_0)\|_{  L_x^2 H_y^{1-\epsilon_0}} + \big\||u|^2 u - |\tilde{u}|^2 \tilde{u}\big\|_{L_t^\frac43 L_x^\frac43 H_y^{1-\epsilon_0} ([t_0,t_j]\times \mathbb{R}^2 \times \mathbb{T})} + \|e\|_{L_t^\frac43 L_x^\frac43 H_y^{1-\epsilon_0}   ([t_0,t_j]\times \mathbb{R}^2 \times \mathbb{T})}\\
\lesssim  &\  M' + \sum\limits_{k=0}^{j-1} C(k)\epsilon + \epsilon.
\end{align*}
Similarly, by Strichartz, \eqref{eq3.33}, \eqref{eq3.34} and the inductive hypothesis,
\begin{align*}
& \|e^{i(t-t_j)\Delta} (u(t_j) - \tilde{u}(t_j))\|_{L_t^4 L_x^4 H_y^{1-\epsilon_0}  (I_j\times \mathbb{R}^2\times \mathbb{T})}\\
 \lesssim &  \|e^{i(t-t_0)\Delta} (u(t_0) - \tilde{u}(t_0)) \|_{L_t^4 L_x^4 H_y^{1-\epsilon_0}  (I_j \times \mathbb{R}^2 \times \mathbb{T})}
+ \|e\|_{L_t^\frac43 L_x^\frac43 H_y^{1-\epsilon_0}   ([t_0,t_j]\times \mathbb{R}^2 \times \mathbb{T})}\\
& \quad  + \big\||u|^2 u - |\tilde{u}|^2 \tilde{u}\big\|_{L_t^\frac43 L_x^\frac43 H_y^{1-\epsilon_0}  ([t_0,t_j] \times \mathbb{R}^2 \times \mathbb{T})}\\
 \lesssim & \epsilon + \sum\limits_{k=0 }^{j-1}  C(k) \epsilon.
\end{align*}
Choosing $\epsilon_1$ sufficiently small depending on $J, M$ and $M'$ to guarantee the hypotheses of Lemma \ref{le3.6} continue to hold as $j$ varies.
\end{proof}

\begin{remark}[Persistence of regularity]
The results in the above theorems can be extended to $H^1(\mathbb{R}^2\times \mathbb{T})$.
\end{remark}
The following theorem implies that it sufficies to show the finiteness of the solution in $L_{t,x}^4 H_y^{1-\epsilon_0}$ for the scattering of \eqref{eq1.1} in $H^1$.
\begin{theorem}[Scattering norm]\label{th2.946}
Suppose that $u\in C_t^0 H_{x,y}^{ 1}(\mathbb{R} \times \mathbb{R}^2_x\times \mathbb{T}_y)$ is a global solution of \eqref{eq1.1} satisfying
$\|u\|_{L_t^4 L_x^4 H_y^{1-\epsilon_0}  (\mathbb{R} \times \mathbb{R}^2_x\times \mathbb{T}_y)} \le L$ and $\|u(0)\|_{H_{x,y}^1} \le M$ for some positive constants $M,\, L$,
then $u$ scatters in $H^1_{x,y}(\mathbb{R}^2\times \mathbb{T})$.
That is,
there exist $u_\pm \in H_{x,y}^{ 1}(\mathbb{R}^2 \times \mathbb{T})$ such that
\begin{equation} \label{eq1.3}
\|u(t,x,y)- e^{it\Delta_{\mathbb{R}^2 \times \mathbb{T}}} u_\pm(x,y)\|_{H_{x,y}^{ 1}} \to 0, \ \text{ as } t\to \pm \infty.
\end{equation}
\end{theorem}
\begin{proof}
By the classical scattering theory as in \cite{C,T2}, we only need to show
\begin{align}\label{eq2.2046}
\|u\|_{L_{t,x}^4 H_y^{1 } \cap L_t^4 W_x^{1,4} L_y^2(\mathbb{R} \times \mathbb{R}^2 \times \mathbb{T})} \le C(M,L).
\end{align}
By Theorem \ref{th2.3}, it suffices to prove \eqref{eq2.2046} as an a priori bound.

Divide the time interval $\mathbb{R}$ into $N \sim (1+ \frac{L}\delta)^{10}$ subintervals $I_j = [t_j,t_{j+1}]$ such that
\begin{align}\label{eq2.2146}
\|u\|_{L_{t,x}^4 H_y^{1-\epsilon_0}(I_j \times \mathbb{R}^2 \times \mathbb{T})} \le \delta,
\end{align}
where $\delta >0 $ will be chosen later.

On each $I_j$, by Strichartz estimate, Sobolev embedding and \eqref{eq2.2146}, we have
\begin{align*}
  \|u\|_{L_t^4 W_x^{1,4} L_y^2 \cap L_{t,x}^4 H_y^1(I_j \times \mathbb{R}^2 \times \mathbb{T})}
\le  & C\Big(\|u(t_j)\|_{H^1} + \left\||u|^2 u \right\|_{L_{t,x}^\frac43 H_y^1 } + \left\||u|^2 u \right\|_{L^\frac43_t W_x^{1,\frac43}L_y^2} \Big)\\
\le  & C\Big(\|u(t_j)\|_{H^1} + \|u\|_{L_{t,x}^4 H_y^{1-\epsilon_0}}^2 \|u\|_{L_{t,x}^4 H_y^1 \cap L_t^4 W_x^{1,4}L_y^2}\Big)\\
\le &  C\left(\|u(t_j)\|_{H^1} + \delta^2 \|u\|_{L_{t,x}^4 H_y^1 \cap L_t^4 W_x^{1,4}L_y^2}\right),
\end{align*}
choosing $\delta \le (\frac1{2C})^\frac12$, we have
\begin{align*}
\|u\|_{L_t^4 W_x^{1,4} L_y^2 \cap L_{t,x}^4 H_y^1(I_j \times \mathbb{R}^2 \times \mathbb{T})} \le 2C\|u(t_j)\|_{H_{x,y}^1}.
\end{align*}
The bound now follows by adding up the bounds on each subintervals $I_j$, which gives \eqref{eq2.2046}.
\end{proof}

\section{Profile decomposition}\label{se3}
In this section, we will establish the linear profile decomposition in $H^1(\mathbb{R}_x^2 \times \mathbb{T}_y)$
in Subsection \ref{sse3.1}, which will heavily rely on the linear profile decomposition in $L^2(\mathbb{R}^2)$.
The linear profile decomposition in $L^2(\mathbb{R}^2)$ for the mass-critical nonlinear Schr\"odinger equation
is established by F. Merle and L. Vega \cite{MV}. After the work of R. Carles and S. Keraani \cite{CK} on the one-dimensional case,
P. B\'egout and A. Vargas \cite{BV} establish the linear profile decomposition of the mass-critical nonlinear Schr\"odinger equation
for arbitrary dimensions by the refined Strichartz inequality \cite{Bo1} and bilinear restriction estimate \cite{T1}.
We also refer to \cite{Killip-Visan1} for a version of the proof of the linear profile decomposition.
We then analyze the nonlinear profiles in Subsection \ref{sse3.2}, which is important to show the existence of the almost-periodic solution in Section \ref{se4}.
\subsection{Linear profile decomposition}\label{sse3.1}
In this subsection, we will establish the linear profile decomposition. The linear profile decomposition is mainly inspired by the mass-critical profile decomposition in $\mathbb{R}^2$,
which we refer to \cite{BV,Bo1,CK,Killip-Visan1,MV}.
\begin{proposition}[Linear profile decomposition in $H^1( \mathbb{R}^2 \times \mathbb{T})$] \label{pro3.9v23}
Let $\{u_n\}_{n\ge 1}
$
be a bounded sequence in $H_{x,y}^{ 1}(\mathbb{R}^2 \times \mathbb{T})$. Then after passing to a subsequence if necessary,
there exist $J^* \in \{0,1, \cdots\} \cup \{\infty\}$,
functions $\phi^{j}$ in $L_x^2 H_{y}^{ 1}(\mathbb{R}^2 \times \mathbb{T})$ and mutually orthogonal frames $(\lambda_n^j, t_n^j, x_n^j,\xi_n^j )_{n\ge 1} \subseteq (0,\infty) \times \mathbb{R} \times \mathbb{R}^2 \times \mathbb{R}^2$,
which means
\begin{align}\label{eq3.141}
\frac{\lambda_n^j}{\lambda_n^k} + \frac{\lambda_n^k}{\lambda_n^j} + \lambda_n^j \lambda_n^{k} |\xi_n^j - \xi_n^{k}|^2  + \frac{|x_n^j- x_n^k |^2 }{\lambda_n^j \lambda^k_n}
+ \frac{|(\lambda_n^j)^2 t_n^j - (\lambda_n^k)^2 t_n^k |}{\lambda_n^j \lambda_n^k} \to \infty, \text{ as } n \to \infty, \text{ for } j\ne k,
\end{align}
with $\lambda_n^j\to 1$ or $\infty$, as $n\to \infty$, $|\xi_n^j| \le C_j$, for every $1\le j \le J$,
and for every $J \le J^*$ a sequence $r_n^J \in L_x^2 H_{y}^{ 1}(\mathbb{R}^2 \times \mathbb{T})$, such that
\begin{align*}
u_n(x,y)  = \sum\limits_{j=1}^J \frac1{\lambda_n^j} e^{ix\xi_n^j} \left(e^{it_n^j\Delta_{\mathbb{R}^2   }} \phi^j\right)\left(\frac{x-x_n^j}{\lambda_n^j},y\right)  + r_n^J(x,y).
\end{align*}
Moreover,
\begin{align}
\lim\limits_{n\to \infty} \left(\|u_n\|_{L_x^2 H_y^{ 1}}^2 - \sum\limits_{j=1}^J \left\|   \phi^j   \right\|_{L_x^2 H_y^{ 1}}^2 - \|r_n^J\|_{L_x^2 H_y^{ 1}}^2 \right) & = 0,\label{eq5.1} \\
\lambda_n^j   e^{-it_n^j \Delta_{\mathbb{R}^2  }}\left(e^{-i(\lambda_n^j x + x_n^j) \xi_n^j}  r_n^J\left(\lambda_n^j x+ x_n^j,y\right)\right) & \rightharpoonup 0 \text{ in } L_x^2 H^{ 1}_{y},  \text{ as  } n\to \infty,\text{ for each }  j\le J, \label{eq5.2} \\
 \limsup\limits_{n\to \infty} \|e^{it \Delta_{\mathbb{R}^2\times \mathbb{T}   }} r_n^{J} \|_{L_t^4 L_x^4 H_y^{1-\epsilon_0} (\mathbb{R}\times \mathbb{R}^2 \times \mathbb{T})} &  \to 0, \text{ as } J\to J^*. \label{eq5.3}
\end{align}
\end{proposition}
\begin{proof}
By Proposition \ref{pr4.1636} and Remark \ref{re3.441}, we have the decomposition
\begin{align*}
u_n(x,y) =  \sum_{j=1}^J \frac1{\lambda_n^j} e^{ix\xi_n^j} (e^{it_n^j \Delta_{\mathbb{R}^2} } \phi^j)\left(\frac{x-x_n^j}{\lambda_n^j}, y-y_n^j\right) + w_n^J(x,y),
\end{align*}
with $\lambda_n^j \ge 1$, and $|\xi_n^j| \le C_j$, for every $1\le j \le J$, and for $j\ne j'$,
\begin{align*}
\frac{\lambda_n^j}{\lambda_n^{j'}} + \frac{\lambda_n^{j'}}{\lambda_n^j} + \lambda_n^j \lambda_n^{j'} |\xi_n^j - \xi_n^{j'}|^2 + \frac{|x_n^j-x_n^{j'}|^2 } {\lambda_n^j \lambda_n^{j'}} + |y_n^j-y_n^{j'}| + \frac{|(\lambda_n^j)^2 t_n^j -(\lambda_n^{j'})^2 t_n^{j'}|}{\lambda_n^j \lambda_n^{j'}} \to \infty, \text{ as } n \to \infty.
\end{align*}
Since $\mathbb{T}$ is compact, we may assume $y_n^j \to y_*^j$, as $n\to \infty$. Then we may replace $\phi^j(\cdot, y) $ by $\phi^j(\cdot, y-y_*^j)$,
and set $y_n^j \equiv 0$.
If $\lambda_n^j$ does not converge to $\infty$, suppose $\lambda_n^j \to \lambda_\infty^j \in [1,\infty)$. Thus, we may replace $\phi^j(x,\cdot)$ by
$\frac1{\lambda_\infty^j} \phi^j\left(\frac{x}{\lambda_\infty^j}, \cdot\right)$ and set $\lambda_n^j \equiv 1$, whilst retaining the conclusion of
Proposition \ref{pro3.9v23}.
\end{proof}
Before presenting Proposition \ref {pr4.1636}, that is the linear profile decomposition in $L_x^2 H_y^{1-\frac{\epsilon_0}2} (\mathbb{R}^2 \times \mathbb{T})$.
We will first establish the refined Strichartz estimate in Proposition \ref{pr4.2436}. We will collect some basic facts appeared in \cite{Killip-Visan1}.
\begin{definition}\label{de4.2136}
Given $j\in \mathbb{Z}$, we write $\mathcal{D}_j$ for the set of all dyadic cubes of side-length $2^j$ in $\mathbb{R}^2$,
\begin{align*}
\mathcal{D}_j = \left\{ [2^j k_1, 2^j(k_1+1)) \times [2^j k_2, 2^j(k_2+1)): (k_1,k_2) \in \mathbb{Z}^2 \right\}.
\end{align*}
We also write $\mathcal{D} = \bigcup\limits_{j} \mathcal{D}_j$.
Given $Q\in \mathcal{D}$, we define $f_Q(x,y)$ by $\mathcal{F}_x  ({f}_Q) = \chi_Q \,\mathcal{F}_x {f} $.
\end{definition}
By the bilinear Strichartz estimate on $\mathbb{R}^2$ \cite{T1} and interpolation, we have
\begin{corollary} \label{co4.2236}
Suppose $Q,Q'\in \mathcal{D}$ with $dist(Q,Q') \gtrsim diam(Q) = diam(Q')$, then for some $1<  p < 2$,
\begin{align*}
\|e^{it\Delta_{\mathbb{R}^2 }} f_Q \cdot e^{it\Delta_{\mathbb{R}^2}} f_{Q'}\|_{L_{t,x}^\frac{11}6}
\lesssim |Q|^{\frac{10}{11}-\frac2p} \|\mathcal{F}_x {f}\|_{L_\xi^p(Q)} \|\mathcal{F}_x {f}\|_{L_\xi^p(Q')}.
\end{align*}
\end{corollary}
By Plancherel's theorem, we have
\begin{lemma} \label{leA.9}
Let $\{R_k\}$ be a family of parallelepipeds in $\mathbb{R}^2$ obeying
$\sup\limits_\xi \sum\limits_{k} \chi_{\alpha R_k}(\xi) \lesssim 1$ for some $\alpha >1$.
Then
\begin{align*}
\left\|\sum_k P_{R_k} f_k\right\|_{L^2}^2 \lesssim \sum_k \|f_k\|_{L^2}^2,
\text{ for any } \{f_k\}\subseteq L^2(\mathbb{R}^2).
\end{align*}
\end{lemma}
\begin{proposition}[Refined Strichartz estimate] \label{pr4.2436}
\begin{align*}
\|e^{it\Delta_{\mathbb{R}^2}} f\|_{L_{t,x,y}^4(\mathbb{R}  \times \mathbb{R}^2 \times \mathbb{T})} \lesssim
\|f\|_{L_{x}^2 H_y^{1-\frac{\epsilon_0}2}}^\frac34 \left(\sup_{Q\in \mathcal{D}} |Q|^{-\frac3{22}} \|e^{it\Delta_{\mathbb{R}^2}} f_Q \|_{L_{t,x,y}^\frac{11}2}\right)^\frac14.
\end{align*}
\end{proposition}
\begin{proof}
Given distinct $\xi, \xi' \in \mathbb{R}^2$, there is a unique maximal pair of dyadic cubes $Q\ni \xi$ and $Q'\ni \xi'$ obeying
$|Q| = |Q'|$ and $4 diam(Q) \le dist(Q,Q') $.

Let $\mathcal{W}$ denote the family of all such pairs as $\xi \ne \xi'$ vary over $\mathbb{R}^2$.
According to this definition,
\begin{align}\label{eq4.4436}
\sum_{(Q,Q')\in \mathcal{W}} \chi_Q(\xi) \chi_{Q'}(\xi') = 1, \text{ for a.e. } (\xi,\xi')\in \mathbb{R}^2 \times \mathbb{R}^2.
\end{align}
Note that since $Q$ and $Q'$ are maximal, $dist(Q,Q') \le 10 diam (Q)$. In addition, this shows that given $Q$ there are a bounded number of
$Q'$ so that $(Q,Q') \in \mathcal{W}$, that is,
\begin{align}\label{eq4.4536}
\forall\, Q\in  \mathcal{W}, \ \sharp \{Q': (Q,Q') \in \mathcal{W}\} \lesssim 1.
\end{align}
In view of \eqref{eq4.4436}, we can write
\begin{align*}
(e^{it\Delta_{\mathbb{R}^2}} f)^2 = \sum_{(Q,Q')\in \mathcal{W}} e^{it\Delta_{\mathbb{R}^2}} f_Q \cdot e^{it\Delta_{\mathbb{R}^2}} f_{Q'}.
\end{align*}
We have the support of the partial space-time transform of $e^{it\Delta_{\mathbb{R}^2}} f_Q \cdot e^{it\Delta_{\mathbb{R}^2}} f_{Q'}$ satisfies
\begin{align}\label{eq4.4637}
supp_{\omega,\xi} \left(\mathcal{F}_{t,x} \left(e^{it\Delta_{\mathbb{R}^2}} f_Q \cdot e^{it\Delta_{\mathbb{R}^2}} f_{Q'}\right)(\omega,\xi)\right) \subseteq R(Q+ Q'),
\end{align}
where $Q+Q'$ denotes the Minkowski sum and
\begin{align*}
R(Q+ Q') = \left\{ (\omega, \eta) : \eta \in Q+ Q', 2 \le \frac{\omega- \frac12 |c(Q+ Q')|^2 - c(Q+ Q')(\eta- c(Q+Q'))}{diam(Q+Q')^2} \le 19 \right\},
\end{align*}
where $c(Q+ Q')$ denotes the center of the cube $Q+Q'$.
We also remind that $diam(Q+Q') = diam(Q) + diam(Q')= 2diam(Q)$.

We need to control the overlap of the Fourier supports, or rather, of the enclosing parallelepipeds.
By \cite{Killip-Visan1}, for any $\alpha \le 1.01$,
\begin{align}\label{eq4.4737}
\sup_{\omega,\eta} \sum_{(Q,Q')\in \mathcal{W}}  \chi_{\alpha  {R}(Q+ Q')} (\omega,\eta) \lesssim 1,
\end{align}
where $\alpha  {R}$ denotes the $\alpha-$dilate of $R$ with the same center.

Similar to the argument in \cite{Killip-Visan1}, we may apply Lemma \ref{leA.9}, H\"older's inequality, Corollary \ref{co4.2236} and \eqref{eq4.4536}
as follows:
\begin{align*}
\|e^{it\Delta_{\mathbb{R}^2}} f\|_{L_{t,x,y}^4}^4
= &¡¡\Big\|\sum_{(Q,Q')\in \mathcal{W}} e^{it\Delta_{\mathbb{R}^2}} f_Q \cdot e^{it\Delta_{\mathbb{R}^2}} f_{Q'} \Big\|_{L_{t,x,y}^2}^2\\
\lesssim &
 \sum_{(Q,Q')\in \mathcal{W}} \|e^{it\Delta_{\mathbb{R}^2}} f_Q \cdot e^{it\Delta_{\mathbb{R}^2}} f_{Q'} \|_{L_{t,x,y}^2}^2\\
\lesssim &  \sum_{(Q,Q')\in \mathcal{W}} \|e^{it\Delta_{\mathbb{R}^2}} f_Q \|_{L_{t,x,y}^\frac{11}2}^\frac12 \|e^{it\Delta_{\mathbb{R}^2}} f_{Q'}\|_{L_{t,x,y}^\frac{11}2}^\frac12 \Big\| \|e^{it\Delta_{\mathbb{R}^2}} f_Q \cdot e^{it\Delta_{\mathbb{R}^2}} f_{Q'}\|_{L_{t,x }^\frac{11}6} \Big\|_{L_{y}^\frac{11}6}^\frac32\\
\lesssim & \Big(\sup_{Q\in \mathcal{D}} |Q|^{-\frac3{22}} \|e^{it\Delta_{\mathbb{R}^2}} f_Q\|_{L_{t,x,y}^\frac{11}2} \Big)
\sum_{(Q,Q')\in \mathcal{W}} \Big(|Q|^{-\frac{2-p}p}
\Big\|\|\mathcal{F}_x {f}(\xi,y)\|_{L_\xi^p(Q)}\|\mathcal{F}_x {f}(\xi,y)\|_{L_\xi^p(Q')} \Big\|_{ L_y^\frac{11}6( \mathbb{T})}\Big)^\frac32\\
\lesssim & \Big(\sup_{Q\in \mathcal{D}} |Q|^{-\frac3{22}} \|e^{it\Delta_{\mathbb{R}^2}} f_Q\|_{L_{t,x,y}^\frac{11}2} \Big)
\sum_{Q\in \mathcal{D}} \Big(|Q|^{-\frac{2-p}p}
\|\mathcal{F}_x {f}(\xi,y)\|_{L_\xi^p L_y^\frac{11}3(Q \times \mathbb{T})}^2\Big)^\frac32, \text{ where } 1< p < 2.
\end{align*}
As in \cite{Killip-Visan1}, we have
\begin{align*}
\sum_{Q\in \mathcal{D}} \Big(|Q|^{-\frac{2-p}p}
\|\mathcal{F}_x {f}(\xi,y)\|_{L_\xi^p L_y^\frac{11}3(Q \times \mathbb{T})}^2\Big)^\frac32
\lesssim \|f\|_{L_{x}^2 H_y^{1-\frac{\epsilon_0}2}}^3.
\end{align*}
In fact, we now assume $\|f\|_{L_x^2 H_y^{1-\frac{\epsilon_0}2}} = 1$ and break
\begin{align*}
\|\mathcal{F}_x {f}(\xi,y)\|_{L_y^\frac{11}3(\mathbb{T})}  = \|\mathcal{F}_x {f}(\xi,y)\|_{L_y^\frac{11}3(\mathbb{T})} \chi_{\big\{\left\|\mathcal{F}_x {f}(\xi,y)\right\|_{L_y^\frac{11}3(\mathbb{T})}\ge 2^{-j}\big\}} (\xi)
+\|\mathcal{F}_x {f}(\xi,y)\|_{L_y^\frac{11}3(\mathbb{T})} \chi_{\big\{\left\|\mathcal{F}_x {f}(\xi,y)\right\|_{L_y^\frac{11}3(\mathbb{T})}\le 2^{-j}\big\}} (\xi),
\end{align*}
where the side-length of $Q$ is $2^{j}$.
We see by H\"older, Sobolev, Plancherel,
\begin{align*}
& \sum_{j\in \mathbb{Z}} \sum_{Q\in \mathcal{D}_j} \Big(|Q|^{-\frac{2-p}p} \big\|\|\mathcal{F}_x {f}(\xi,y)\|_{L_y^\frac{11}3(\mathbb{T})}
\chi_{\big\{\left\|\mathcal{F}_x {f}(\xi,y)\right\|_{L_y^\frac{11}3(\mathbb{T})}\ge 2^{-j}\big\}} (\xi) \big\|_{L_\xi^p(Q)}^2 \Big)^\frac32\\
\lesssim & \Big(\sum_{j\in \mathbb{Z}} \sum_{Q\in \mathcal{D}_j}  |Q|^{-1+\frac{p}2} \big\|\|\mathcal{F}_x {f}(\xi,y)\|_{L_y^\frac{11}3(\mathbb{T})}
\chi_{\big\{\left\|\mathcal{F}_x {f}(\xi,y)\right\|_{L_y^\frac{11}3(\mathbb{T})}\ge 2^{-j}\big\}} (\xi) \big\|_{L_\xi^p(Q)}^p\Big)^\frac3p\\
\lesssim & \Big(\int_{\mathbb{R}^2} \sum_{j:\  \|\mathcal{F}_x {f}(\xi,y)\|_{L_y^\frac{11}3 \ge 2^{-j}}} 2^{-j(2-p)} \|\mathcal{F}_x {f}(\xi,y)
\|_{L_y^\frac{11}3}^p \,\mathrm{d}\xi\Big)^\frac3p\\
\lesssim & \Big(\int_{\mathbb{R}^2} \|\mathcal{F}_x {f}(\xi,y)\|_{L_y^\frac{11}3}^2 \,\mathrm{d}\xi\Big)^\frac3p
\lesssim
\|f\|_{L_x^2 H_y^{1-\frac{\epsilon_0}2} }^\frac6p \lesssim 1.
\end{align*}
Similarly, by H\"older, Sobolev, Plancherel, we have
\begin{align*}
& \sum_{j\in \mathbb{Z}} \sum_{Q\in \mathcal{D}_j} \Big(|Q|^{-\frac{2-p}p} \big\|\|\mathcal{F}_x {f}(\xi,y)\|_{L_y^\frac{11}3(\mathbb{T})}
\chi_{\big\{\left\|\mathcal{F}_x {f}(\xi,y)\right\|_{L_y^\frac{11}3(\mathbb{T})}\le 2^{-j}\big\}} (\xi) \big\|_{L_\xi^p(Q)}^2 \Big)^\frac32\\
\lesssim & \sum_{j\in \mathbb{Z}} \sum_{Q\in \mathcal{D}_j} |Q|^\frac12 \big\|\|\mathcal{F}_x {f}(\xi,y)\|_{ L_y^\frac{11}3(\mathbb{T})}
\chi_{\big\{\left\|\mathcal{F}_x {f}(\xi,y)\right\|_{L_y^\frac{11}3(\mathbb{T})}\le 2^{-j}\big\}} (\xi) \big\|_{L_\xi^3(Q)}^3         \\
\lesssim & \int_{\mathbb{R}^2} \sum_{j: \ \|\mathcal{F}_x {f}(\xi,y)\|_{L_y^\frac{11}3} \le 2^{-j}}
 (2^{-j})^{-1} \|\mathcal{F}_x {f}(\xi,y)\|_{L_y^\frac{11}3}^3\, \mathrm{d}\xi\\
  \lesssim  & \|\mathcal{F}_x {f}(\xi,y)\|_{L_\xi^2 L_y^\frac{11}3 }^2
  \lesssim
\|f\|_{L_x^2 H_y^{1-\frac{\epsilon_0}2} }^2 \lesssim 1.
\end{align*}
Therefore, the proof is completed.
\end{proof}
To prove the inverse Strichartz estimate, we also need the following facts:
\begin{lemma}[Refined Fatou]\label{leA.5}
Suppose $\{f_n\} \subseteq L^4(\mathbb{R}^3 \times \mathbb{T})$ with $\limsup\limits_{n\to \infty} \|f_n\|_{L^4} < \infty$. If $f_n\to f$ almost everywhere, then
\begin{align*}
\|f_n\|_{L^4}^4 - \|f_n-f\|_{L^4}^4 -\|f\|_{L^4}^4 \to 0, \text{ as } n\to \infty.
\end{align*}
\end{lemma}
\begin{proposition}[Local smoothing estimate] \label{pr4.1436}
Fix $\epsilon > 0$, we have
\begin{align*}
\int_{\mathbb{R}} \int_{\mathbb{R}^2\times \mathbb{T}} \left|\left(|\nabla_x |^\frac12 e^{it\Delta_{\mathbb{R}^2}} f\right)(x,y)\right|^2 \langle x\rangle^{-1-\epsilon} \,\mathrm{d}x\mathrm{d}y \mathrm{d}t \lesssim_\epsilon \|f\|_{L_{x,y}^2(\mathbb{R}^2 \times \mathbb{T})}^2.
\end{align*}
Furthermore, if $\epsilon \ge 1$, we have
\begin{align*}
\int_{\mathbb{R}} \int_{\mathbb{R}^2\times \mathbb{T}} \left|\left(|\langle \nabla_x  \rangle |^\frac12 e^{it\Delta_{\mathbb{R}^2}} f \right)(x,y)\right|^2 \langle x\rangle^{-1-\epsilon} \,\mathrm{d}x\mathrm{d}y \mathrm{d}t \lesssim_\epsilon \|f\|_{L_{x,y}^2(\mathbb{R}^2 \times \mathbb{T})}^2.
\end{align*}
\end{proposition}
We can now prove the inverse Strichartz estimate.
\begin{proposition}[Inverse Strichartz estimate] \label{pr4.2536}
For $\{f_n\} \subseteq L_x^2 H_y^{1-\frac{\epsilon_0}2} $ satisfying  
\begin{align}\label{eq3.948}
\lim_{n\to \infty} \|f_n\|_{L_x^2H_y^{1-\frac{\epsilon_0}2} } = A \text{ and } \lim_{n\to \infty} \|e^{it\Delta_{\mathbb{R}^2}} f_n\|_{L_{t,x,y}^4} = \epsilon,
\end{align}
there exist a subsequence in $n$, $\phi \in L_x^2 H_y^{1-\frac{\epsilon_0}2} $, $\{\lambda_n\} \subseteq (0,\infty)$, $\xi_n\in \mathbb{R}^2$,
and $(t_n,x_n,y_n)\in \mathbb{R}\times \mathbb{R}^2 \times \mathbb{T}$ so that along the subsequence, we have the following:
\begin{align}
& \lambda_n e^{-i\xi_n(\lambda_n x+ x_n) } (e^{it_n\Delta_{\mathbb{R}^2}} f_n)(\lambda_n x + x_n,y+y_n) \rightharpoonup  \phi(x,y)\  \text{ in } L_x^2 H_y^{1-\frac{\epsilon_0}2} , \text{ as } n\to \infty, \\
 & \lim_{n\to \infty} \left(\|f_n(x,y)\|_{L_x^2 H_y^{1-\frac{\epsilon_0}2}}^2 - \|f_n-\phi_n\|_{L_x^2 H_y^{1-\frac{\epsilon_0}2}}^2 \right) = \|\phi\|
_{L_x^2 H_y^{1-\frac{\epsilon_0}2}}^2
\gtrsim A^2 \left(\frac\epsilon A\right)^{24},
\label{eq4.4936} \\
&  \limsup_{n\to \infty} \|e^{it\Delta_{\mathbb{R}^2}} (f_n-\phi_n) \|_{L_{t,x,y}^4}^4 \le \epsilon^4 \left( 1- c \left(\frac\epsilon A\right)^\beta\right), \label{eq4.5036}
\end{align}
where $c$ and $\beta$ are constants,
\begin{align*}
\phi_n(x,y) = \frac1{\lambda_n} e^{ix\xi_n} \left(e^{-i\frac{t_n}{\lambda_n^2} \Delta_{\mathbb{R}^2}} \phi\right)\left(\frac{x-x_n}{\lambda_n},y-y_n\right).
\end{align*}
Moreover, if $\{f_n\}$ is bounded in $H_{x,y}^1(\mathbb{R}^2 \times \mathbb{T})$, we can take $\lambda_n \ge 1$, $|\xi_n| \lesssim 1$, with
$\phi \in L_x^2 H_y^{1 }(\mathbb{R}^2 \times \mathbb{T}) $, and
\begin{align}
\lambda_n e^{-i\xi_n(\lambda_n x+ x_n) } (e^{it_n\Delta_{\mathbb{R}^2}} f_n)(\lambda_n x + x_n,y+y_n) \rightharpoonup  \phi(x,y)\  \text{ in } L_x^2 H_y^{1 } , \text{ as } n\to \infty, \label{eq3.1248}\\
\lim_{n\to \infty} \left(\|f_n(x,y)\|_{L_x^2 H_y^{1 }}^2 - \|f_n-\phi_n\|_{L_x^2 H_y^{1 }}^2 \right) = \|\phi\|
_{L_x^2 H_y^{1 }}^2
\gtrsim A^2 \left(\frac\epsilon A\right)^{24}.\label{eq3.1348}
\end{align}

\end{proposition}
\begin{proof}
By Proposition \ref{pr4.2436}, there exists $\{Q_n\} \subseteq \mathcal{D}$ so that
\begin{align}\label{eq4.5236}
\epsilon^4 A^{-3} \lesssim \liminf_{n\to \infty} |Q_n|^{-\frac3{22}} \|e^{it\Delta_{\mathbb{R}^2}} (f_n)_{Q_n}\|_{L_{t,x,y}^\frac{11}2}.
\end{align}
Let $\lambda_n $ to be the inverse of the side-length of $Q_n$, which implies $|Q_n|= \lambda_n^{-2}$.
We also set $\xi_n = c(Q_n)$, which is the center of the cube. 
By H\"older's inequality, we have
\begin{align}
\liminf_{n\to \infty} |Q_n|^{-\frac3{22}} \|e^{it\Delta_{\mathbb{R}^2}} (f_n)_{Q_n} \|_{L_{t,x,y}^\frac{11}2}
& \lesssim \liminf_{n\to \infty} |Q_n|^{-\frac3{22}} \|e^{it\Delta_{\mathbb{R}^2}} (f_n)_{Q_n} \|_{L_{t,x,y}^4}^\frac8{11} \|e^{it\Delta_{\mathbb{R}^2}} (f_n)_{Q_n} \|_{L_{t,x,y}^\infty}^\frac3{11} \nonumber \\
&¡¡\lesssim \liminf_{n\to \infty} \lambda_n^\frac3{11} \epsilon^\frac8{11} \|e^{it\Delta_{\mathbb{R}^2}} (f_n)_{Q_n} \|_{L_{t,x,y}^\infty}^\frac3{11}.
\end{align}
Thus by \eqref{eq4.5236}, there exists $(t_n,x_n,y_n) \in \mathbb{R} \times \mathbb{R}^2 \times \mathbb{T}$ so that
\begin{align}\label{eq4.5336}
\liminf_{n\to \infty} \lambda_n |(e^{it_n \Delta_{\mathbb{R}^2}} (f_n)_{Q_n} )(x_n,y_n)|\gtrsim \epsilon^{12} A^{-11}.
\end{align}
By the weak compactness of $L_x^2 H_y^{1-\frac{\epsilon_0}2}$, we have
\begin{align*}
\lambda_n e^{-i\xi_n(\lambda_n x+x_n)} (e^{it_n \Delta_{\mathbb{R}^2}} f_n)(\lambda_n x + x_n, y + y_n) \rightharpoonup  \phi(x,y) \text{ in } L_x^2 H_y^{1-\frac{\epsilon_0}2}, \text{ as } n\to \infty.
\end{align*}
Define $h$ so that $\mathcal{F}_x {h}$
is the characteristic function of the cube $[-\frac12,\frac12]^2$. Then $h(x) \delta_0(y)\in L_x^2 H_y^{-1+\frac{\epsilon_0}2}(\mathbb{R}^2 \times \mathbb{T})$.
From \eqref{eq4.5336}, we obtain
\begin{align}\label{eq4.5436}
& |\langle h(x) \delta_0(y), \phi(x,y)\rangle_{x,y}| \nonumber \\
= & \lim_{n\to \infty} \left|\int_{\mathbb{R}^2 \times \mathbb{T}} \delta_0(y) \bar{h}(x) \lambda_n e^{-i\xi_n(\lambda_n x + x_n)} (e^{it_n \Delta_{\mathbb{R}^2}} f_n)(\lambda_n x +x_n, y+ y_n) \,\mathrm{d}x \mathrm{d}y \right| \nonumber \\
=  & \lim_{n\to \infty} \lambda_n |(e^{it_n\Delta_{\mathbb{R}^2}} (f_n)_{Q_n})(x_n,y_n)|
\gtrsim \epsilon^{12} A^{-11},
\end{align}
which implies \eqref{eq4.4936}, and $\phi$ carries non-trivial norm in $ L_x^2 H_y^{1-\frac{\epsilon_0}2}$. 
We are left to verify \eqref{eq4.5036}.
By Proposition \ref{pr4.1436} and the Rellich-Kondrashov theorem, we have
\begin{align*}
e^{it\Delta_{\mathbb{R}^2}} \left( \lambda_n e^{-i \xi_n(\lambda_n x + x_n) } (e^{it_n\Delta_{\mathbb{R}^2}} f_n)(\lambda_n x + x_n, y+ y_n)\right) \to
e^{it\Delta_{\mathbb{R}^2}} \phi(x,y),\text{ a.e.  } (t, x,y) \in \mathbb{R} \times \mathbb{R}^2\times \mathbb{T}, \text{ as $n\to \infty$}.
\end{align*}
Then by applying Lemma \ref{leA.5},
 we obtain
\begin{align*}
\|e^{it \Delta_{\mathbb{R}^2}} f_n\|_{L_{t,x,y}^4}^4 - \|e^{it\Delta_{\mathbb{R}^2}} (f_n-\phi_n)\|_{L_{t,x,y}^4}^4 - \|e^{it\Delta_{\mathbb{R}^2}} \phi_n\|_{L_{t,x,y}^4}^4 \to 0, \text{ as } n\to \infty.
\end{align*}
Together with \eqref{eq4.5436}, we obtain \eqref{eq4.5036}.

If $\{f_n\}$ is bounded in $H^1_{x,y}(\mathbb{R}^2\times \mathbb{T})$, 
we have
\begin{align*}
\limsup\limits_{n\to \infty} \|P_{\ge R}^x f_n\|_{L_x^2 H_y^{1-\frac{\epsilon_0}2}}
 & \lesssim 
\limsup\limits_{n\to \infty} \langle R\rangle^{-\frac{\epsilon_0}2} \|f_n\|_{H_{x,y}^1} \to 0, \text{ as  }  R\to \infty.
\end{align*}
For $R\in 2^{\mathbb{Z}}$ large enough depending on $A$ and $\epsilon$, by \eqref{eq3.948}, $P_{\le R}^x f_n$ satisfies
\begin{align*}
\lim\limits_{n\to \infty} \|P_{\le R}^x f_n\|_{L_x^2 H_y^{1-\frac{\epsilon_0}2}}  & \ge \lim\limits_{n\to \infty} \|f_n\|_{L_x^2 H_y^{1-\frac{\epsilon_0}2}}
- \lim\limits_{n\to \infty} \|P_{\ge R}^x f_n\|_{L_x^2 H_y^{1-\frac{\epsilon_0}2}} \ge \frac12 A, 
\end{align*} 
and
\begin{align*}
\lim\limits_{n\to \infty} \|e^{it\Delta_{\mathbb{R}^2}} P_{\le R}^x f_n\|_{L_{t,x,y}^4} & \ge \lim\limits_{n\to \infty} \|e^{it\Delta_{\mathbb{R}^2}} f_n\|_{L_{t,x,y}^4} - \lim\limits_{n\to \infty} \|e^{it\Delta_{\mathbb{R}^2}} P_{\ge R}^x  f_n\|_{L_{t,x,y}^4}\\
& \ge  \lim\limits_{n\to \infty} \|e^{it\Delta_{\mathbb{R}^2}} f_n\|_{L_{t,x,y}^4} -   \lim\limits_{n\to \infty} C \| P_{\ge R}^x f_n\|_{L_x^2 H_y^{1-\frac{\epsilon_0}2}} \ge \frac\epsilon2.
\end{align*}
So we can replace $f_n$ by $P_{\le R}^x f_n$ in the assumption of the proposition, for $R = R(A,\epsilon) > 0$ large enough, then we can take $\{Q_n\} \subseteq \mathcal{D}$ and
$|Q_n|\lesssim R^2$, then $\lambda_n \gtrsim R^{-1}$, and $|\xi_n| = |c(Q_n)| \lesssim R$.
In the proof above, we see since $\{f_n\}$ is bounded in $H^1_{x,y}(\mathbb{R}^2\times \mathbb{T})$, we can obtain
\begin{align*}
\lambda_n e^{-i\xi_n(\lambda_n x+x_n)} (e^{it_n \Delta_{\mathbb{R}^2}} f_n)(\lambda_n x + x_n, y + y_n) \rightharpoonup  \phi(x,y)
\end{align*}
holds for some $\phi \in L_x^2 H_y^{1 }(\mathbb{R}^2 \times \mathbb{T})$.  
We also see $h(x) \delta_0(y)\in L_x^2 H_y^{-1 }(\mathbb{R}^2 \times \mathbb{T})$, then as \eqref{eq4.5436}, we have \eqref{eq3.1348}.
\end{proof}
By using Proposition \ref{pr4.2536} repeatedly until the remainder has asymptotically trivial linear
evolution in $L_{t,x,y}^4$, we can obtain the following result:
\begin{proposition}[Linear profile decomposition in $L_x^2 H_y^{1-\frac{\epsilon_0}2} (\mathbb{R}^2 \times \mathbb{T})$]  \label{pr4.1636}
Let $\{u_n\}$ be a bounded sequence in $L_x^2 H_y^{1-\frac{\epsilon_0}2}(\mathbb{R}^2 \times \mathbb{T})$. Then (after passing to a subsequence if necessary) there
exists $J^*\in  \{0,1, \cdots \} \cup \{\infty\}$, functions $\phi^j \subseteq L_x^2 H_y^{1-\frac{\epsilon_0}2}$,
$(\lambda_n^j, t_n^j, x_n^j,\xi_n^j )_{n\ge 1} \subseteq (0,\infty) \times \mathbb{R} \times \mathbb{R}^2 \times \mathbb{R}^2$,
so that for any $J\le J^*$, defining $w_n^J$ by
\begin{align}\label{eq3.441}
u_n(x,y) = & \sum_{j=1}^J \frac1{\lambda_n^j} e^{ix\xi_n^j} (e^{it_n^j \Delta_{\mathbb{R}^2  } } \phi^j)\left(\frac{x-x_n^j}{\lambda_n^j}, y-y_n^j\right) + w_n^J(x,y),
\end{align}
we have the following properties:
\begin{align*}
& \limsup_{n\to \infty} \|e^{it\Delta_{\mathbb{R}^2  }} w_n^J\|_{L_{t,x,y}^4(\mathbb{R} \times \mathbb{R}^2 \times \mathbb{T} )} \to 0, \ \text{ as } J\to \infty, \\
& \lambda_n^j e^{-it_n^j \Delta_{\mathbb{R}^2  }} 
\left(e^{-i(\lambda_n^j x + x_n^j) \xi_n^j} w_n^J(\lambda_n^j + x_n^j, y+ y_n^j)\right) \rightharpoonup 0  \text{ in } L_x^2 H_y^{1-\frac{\epsilon_0}2}(\mathbb{R}^2 \times \mathbb{T}),  \text{ as  } n\to \infty,\text{ for each }  j\le J,\\
&\sup_{J} \lim_{n\to \infty} \left(\|u_n\|_{L_x^2 H_y^{1-\frac{\epsilon_0}2}}^2 - \sum_{j=1}^J \|\phi^j\|_{L_x^2 H_y^{1-\frac{\epsilon_0}2}}^2 - \|w_n^J\|_{L_x^2 H_y^{1-\frac{\epsilon_0}2}}^2\right) = 0,
\end{align*}
and lastly, for $j\ne j'$, and $n\to \infty$,
\begin{align*}
\frac{\lambda_n^j}{\lambda_n^{j'}} + \frac{\lambda_n^{j'}}{\lambda_n^j} + \lambda_n^j \lambda_n^{j'} |\xi_n^j - \xi_n^{j'}|^2 + \frac{|x_n^j-x_n^{j'}|^2 } {\lambda_n^j \lambda_n^{j'}} + |y_n^j-y_n^{j'}| + \frac{|(\lambda_n^j)^2 t_n^j -(\lambda_n^{j'})^2 t_n^{j'}|}{\lambda_n^j \lambda_n^{j'}} \to \infty.
\end{align*}
Moreover, if $ u_n $ is bounded in $H^1(\mathbb{R}^2\times \mathbb{T})$, we can take $\lambda_n^j \ge 1$, and $|\xi_n^j| \le C_j$, for every $1\le j \le J$.
We also have $\{\phi^j\}_{j=1}^{J} \subseteq L_x^2 H_y^{1 }$,
\begin{align*}
\lambda_n^j e^{-it_n^j \Delta_{\mathbb{R}^2  }}  
\left(e^{-i(\lambda_n^j x + x_n^j) \xi_n^j} w_n^J(\lambda_n^j + x_n^j, y+ y_n^j)\right)
\rightharpoonup & 0  \text{ in } L_x^2 H_y^{1 },  \text{ as  } n\to \infty,\text{ for each }  j\le J,\\
\sup_{J} \lim_{n\to \infty} \left(\|u_n\|_{L_x^2 H_y^{1 }}^2 - \sum_{j=1}^J \|\phi^j\|_{L_x^2 H_y^{1 }}^2 - \|w_n^J\|_{L_x^2 H_y^{1 }}^2\right) = &  0.
\end{align*}
\end{proposition}
\begin{remark}\label{re3.441}
By using Plancherel, interpolation, the H\"older inequality and Proposition \ref{pr4.1636}, we have
\begin{align*}
  \limsup_{n\to \infty} \|e^{it\Delta_{\mathbb{R}^2\times \mathbb{T} }} w_n^J\|_{L_{t,x}^4 H_y^{1-\epsilon_0}}
  =   & \limsup_{n\to \infty} \|e^{it\Delta_{\mathbb{R}^2 }} w_n^J\|_{L_{t,x}^4 H_y^{1-\epsilon_0}}\\
\lesssim & \limsup_{n\to \infty} \|w_n^J\|_{L_x^2 H_y^{1-\frac{\epsilon_0}2}}^{\frac{2-2\epsilon_0}{2-\epsilon_0}}  \|e^{it\Delta_{\mathbb{R}^2 }} w_n^J\|_{L_{t,x,y}^4}^{\frac{\epsilon_0}{2-\epsilon_0}}  \to 0, \text{ as }  J\to \infty.
\end{align*}
\end{remark}

\subsection{Approximation of profiles}\label{sse3.2}
In this subsection, by using the solution of the resonant Schr\"odinger system in \cite{YZ} to approximate the nonlinear Schr\"odinger
equation with initial data be the bubble in the linear profile decomposition, we can show the nonlinear profile have a bounded space-time norm.
\begin{lemma}[Large-scale profiles ] \label{le5.7}
Let $\phi \in L_x^2 H_y^{ 1}(\mathbb{R}^2\times \mathbb{T})$ be given, then

(i) There is $\lambda_0 = \lambda_0(\phi)$ sufficiently large such that for $\lambda \ge \lambda_0$, there is a unique global solution
$U_\lambda \in C_t^0 L_x^2 H_{y}^{ 1} (\mathbb{R} \times \mathbb{R}^2\times \mathbb{T})$ of
\begin{align}\label{eq3.833}
\begin{cases}
i\partial_t U_\lambda + \Delta_{\mathbb{R}^2 \times \mathbb{T}} U_\lambda =   |U_\lambda|^2 U_\lambda,\\
U_\lambda(0,x,y) = \frac1\lambda   \phi\left(\frac{x}\lambda, y\right) .
\end{cases}
\end{align}
Moreover, for $\lambda \ge \lambda_0$,
\begin{align*}
\|U_\lambda\|_{L_t^\infty L_x^2 H_{y}^{ 1}\cap L_{t,x}^4 H_y^1(\mathbb{R} \times \mathbb{R}^2 \times \mathbb{T})}  \lesssim_{\|\phi\|_{L_x^2 H_y^1}} 1.
\end{align*}
(ii) Assume $\epsilon_1 $ is a sufficiently small positive constant depending only on $\|\phi\|_{L_x^2 H_y^1}, 
 \vec{v}_0 \in  H^2_x h^1$, and
\begin{align}\label{eq3.935}
\|\vec{\phi}- \vec{v}_0\|_{L^2_x h^1 } \le \epsilon_1.
\end{align}

Let $\vec{v} \in   C_t^0 H_x^1  h^1(\mathbb{R} \times \mathbb{R}^2 \times \mathbb{Z})$ denote the solution of
\begin{align}\label{eq3.4new}
\begin{cases}
i\partial_t v_j + \Delta_{\mathbb{R}^2} v_j =  \sum\limits_{(j_1,j_2,j_3)\in \mathcal{R}(j)} v_{j_1} \bar{v}_{j_2} v_{j_3},\\
v_j(0) = v_{0,j}, j\in \mathbb{Z}.
\end{cases}
\end{align}
For $\lambda \ge 1$, we define
\begin{align*}
v_{j,\lambda}(t,x) = \frac1\lambda v_j\left(\frac{t}{\lambda^2}, \frac{x}\lambda\right), \quad (t,x) \in \mathbb{R} \times \mathbb{R}^2 ,\\
V_\lambda(t,x,y) = \sum\limits_{j\in \mathbb{Z}} e^{-it|j|^2} e^{i\langle y, j\rangle} v_{j,\lambda}(t,x), \quad (t,x,y) \in  \mathbb{R} \times \mathbb{R}^2 \times \mathbb{T} .
\end{align*}
Then
\begin{align*}
\limsup\limits_{\lambda\to \infty} \|U_\lambda - V_\lambda\|_{L_t^\infty L_x^2 H_{y}^{ 1} \cap L_{t,x}^4 H_y^1(\mathbb{R} \times \mathbb{R}^2 \times \mathbb{T})} \lesssim_{\|\phi\|_{L_x^2 H_y^1}} \epsilon_1.
\end{align*}
\end{lemma}
\begin{proof}
We will show $V_\lambda$ is an approximate solution to the cubic nonlinear Schr\"odinger equation on $\mathbb{R}^2 \times \mathbb{T}$ 
 in the sense of Theorem \ref{le2.6}.
By noting $v_j$ satisfies \eqref{eq3.4new} and an easy computation, we have
\begin{align}\label{eq3.933}
e_\lambda & = (i\partial_t + \Delta_{\mathbb{R}^2\times \mathbb{T}}) V_\lambda - |V_\lambda|^2 V_\lambda \notag \\
     & = - \sum\limits_{j\in \mathbb{Z}} e^{-it|j|^2} e^{i\langle y,j\rangle} \sum\limits_{ (j_1,j_2,j_3) \in \mathcal{NR}(j) } e^{-it(|j_1|^2 - |j_2|^2 + |j_3|^2 - |j|^2)} (v_{j_1,\lambda} \bar{v}_{j_2,\lambda} v_{j_3,\lambda})(t,x),
\end{align}
where
\begin{equation*}
 \mathcal{NR}(j) = \left\{ (j_1,j_2,j_3) \in \mathbb{Z}^3: j_1-j_2+j_3-j=0,|j_1|^2-|j_2|^2+|j_3|^2-|j|^2 \ne 0 \right\}.
\end{equation*}
We first decompose
$e_\lambda = P_{\ge 2^{-10}}^x e_\lambda + P_{\le 2^{-10}}^x e_\lambda$.
By Berstein's inequality, Plancherel theorem, \eqref{eq3.933}, Leibnitz's rule, and H\"older's inequality, we have
\begin{align}\label{eq3.1033}
& \|P_{\ge 2^{-10}}^x e_\lambda \|_{L_{t }^\frac43  L_x^{\frac43} H_y^1(\mathbb{R} \times \mathbb{R}^2 \times \mathbb{T})}\\
\lesssim & \|P_{\ge 2^{-10}}^x \nabla_x e_\lambda \|_{L_{t,x}^\frac43  H_y^1(\mathbb{R} \times \mathbb{R}^2 \times \mathbb{T})} \notag\\
\lesssim  & \frac1\lambda  \Bigg\| \bigg(\sum_{j\in \mathbb{Z}} \langle j\rangle^2 \bigg|\sum_{ { (j_1,j_2,j_3) \in \mathcal{NR}(j)}  }
e^{-i\lambda^2 t(|j_1|^2 - |j_2|^2 + |j_3|^2 - |j|^2)}   ( v_{j_1 } \cdot \overline{\nabla_x v}_{j_2 }\cdot v_{j_3 })(t,x) \bigg|^2\bigg)^\frac12 \Bigg\|_{L_{t,x}^\frac43}\notag\\
& + \frac1\lambda  \Bigg\| \bigg(\sum_{j\in \mathbb{Z}} \langle j\rangle^2 \bigg|\sum_{ {( j_1,j_2,j_3) \in \mathcal{NR}(j)}  }
e^{-i\lambda^2 t(|j_1|^2 - |j_2|^2 + |j_3|^2 - |j|^2)}   ( v_{j_1 } \cdot \overline{v}_{j_2 }  \cdot \nabla_x  v_{j_3 })(t,x) \bigg|^2\bigg)^\frac12 \Bigg\|_{L_{t,x}^\frac43}\notag\\
& +  \frac1\lambda  \Bigg\| \bigg(\sum_{j\in \mathbb{Z}} \langle j\rangle^2 \bigg|\sum_{ {(j_1,j_2,j_3) \in \mathcal{NR}(j)}  }
e^{-i\lambda^2 t(|j_1|^2 - |j_2|^2 + |j_3|^2 - |j|^2)}   (\nabla_x v_{j_1 } \cdot \overline{v}_{j_2 } \cdot v_{j_3 })(t,x) \bigg|^2\bigg)^\frac12 \Bigg\|_{L_{t,x}^\frac43}
\notag\\
\lesssim  & \frac1\lambda   \Bigg\| \bigg(\sum_{j} \langle j\rangle^2 |\nabla_x v_{j}|^2\bigg)^\frac12    \Bigg\|_{L_{t,x}^4}
\Bigg\| \bigg(\sum_{j} \langle j\rangle^2 |  v_{j}|^2\bigg)^\frac12    \Bigg\|_{L_{t,x}^4}^2 \notag
 \lesssim_{\|\phi\|_{L_x^2 H_y^1}}     \frac1\lambda \|\vec{v}_0 \|_{ H^1_x  h^1}^3.
\end{align}
Thus $P_{\ge 2^{-10}}^x e_\lambda$ is acceptable when $\lambda\ge \lambda_0$. We turn to the estimate of $P_{\le 2^{-10}}^x e_\lambda$. By integrating by parts, we have
\begin{align*}
&   \int_0^t e^{i(t-\tau)\Delta_{\mathbb{R}^2 \times \mathbb{T}}} P_{\le 2^{-10}}^x e_\lambda(\tau) \,\mathrm{d}\tau\\
= &  - \sum_{j\in \mathbb{Z}} \sum_{(j_1,j_2,j_3) \in \mathcal{NR}(j) }
\int_0^t e^{i(t -\tau) ( \Delta_{\mathbb{R}^2} + |j_1|^2 - |j_2|^2 + |j_3|^2 - |j|^2)} e^{-it (|j_1|^2 - |j_2|^2 + |j_3|^2)} e^{i\langle y,j\rangle}
P_{\le 2^{-10}}^x (v_{j_1,\lambda} \bar{v}_{j_2,\lambda} v_{j_3,\lambda})(\tau, x)\,\mathrm{d}\tau
\\
= & -   \sum_{j\in \mathbb{Z}} \sum_{(j_1,j_2,j_3) \in \mathcal{NR}(j)}
 e^{-i t (|j_1|^2 - |j_2|^2 + |j_3|^2)}e^{i\langle y,j\rangle} \frac{ie^{i t ( \Delta_{\mathbb{R}^2} + |j_1|^2 - |j_2|^2 + |j_3|^2 - |j|^2)}}{\Delta_{\mathbb{R}^2} + |j_1|^2 - |j_2|^2 + |j_3|^2 - |j|^2}
 P_{\le 2^{-10}}^x (v_{j_1,\lambda} \bar{v}_{j_2,\lambda} v_{j_3,\lambda})(0)
 \\
&  - \sum_{j\in \mathbb{Z}} \sum_{(j_1,j_2,j_3) \in \mathcal{NR}(j) } e^{-i t (|j_1|^2 - |j_2|^2 + |j_3|^2)}e^{i\langle y,j\rangle} \frac{i}{   \Delta_{\mathbb{R}^2} + |j_1|^2 - |j_2|^2 + |j_3|^2 - |j|^2}
    P_{\le 2^{-10}}^x (v_{j_1,\lambda} \bar{v}_{j_2,\lambda} v_{j_3,\lambda})( t )
\\ & + \sum_{j\in \mathbb{Z}} \sum_{(j_1,j_2,j_3) \in \mathcal{NR}(j)} e^{-i t (|j_1|^2 - |j_2|^2 + |j_3|^2)}e^{i\langle y,j\rangle} i \int_0^t  \frac{ie^{i( t -\tau)( \Delta_{\mathbb{R}^2} + |j_1|^2 - |j_2|^2 + |j_3|^2 - |j|^2)}}{\Delta_{\mathbb{R}^2} + |j_1|^2 - |j_2|^2 + |j_3|^2 - |j|^2} \partial_\tau  P_{\le 2^{-10}}^x (v_{j_1,\lambda} \bar{v}_{j_2,\lambda} v_{j_3,\lambda})(\tau) \, \mathrm{d}\tau
 \\
 =: &  A_1 + A_2 + A_3.
\end{align*}
For $A_1$, by Plancherel's theorem, the boundedness of the operator $\frac{P_{\le 2^{-10}}^x}{\Delta_{\mathbb{R}^2} + |j_1|^2-|j_2|^2 + |j_3|^2 -|j|^2}$ on $L_x^2(\mathbb{R}^2)$ when $(j_1,j_2,j_3)\in \mathcal{NR}(j)$, H\"older's inequality and the Sobolev embedding theorem, we have
\begin{align*}
 &  \Bigg\|\sum_{j\in \mathbb{Z}} \sum_{(j_1,j_2,j_3 ) \in \mathcal{NR}(j) } e^{-it (|j_1|^2 - |j_2|^2 + |j_3|^2)} e^{i\langle y,j\rangle} \frac{ie^{i t (\Delta_{\mathbb{R}^2} + |j_1|^2 - |j_2|^2 + |j_3|^2 - |j|^2)}}{ \Delta_{\mathbb{R}^2} + |j_1|^2 - |j_2|^2 + |j_3|^2 - |j|^2 } P_{\le 2^{-10}}^x (v_{j_1,\lambda} \bar{v}_{j_2,\lambda} v_{j_3,\lambda})(0,x )\Bigg\|_{L_t^\infty L_x^2 H_{y}^{1}
 \cap L_{t,x}^4 H_y^1}\\
\sim &  \Bigg\| \Bigg(\sum_{j\in \mathbb{Z}} \langle j\rangle^2 \Bigg | \sum_{ (j_1,j_2,j_3 ) \in \mathcal{NR}(j) }    \frac{ P_{\le 2^{-10}}^x (v_{j_1,\lambda} \bar{v}_{j_2,\lambda} v_{j_3,\lambda})(0,x )}{ \Delta_{\mathbb{R}^2} + |j_1|^2 - |j_2|^2 + |j_3|^2 - |j|^2 }\Bigg |^2\Bigg)^\frac12 \Bigg\|_{L_x^2} \\
\lesssim &  \bigg\| \Big(\sum_{j\in \mathbb{Z}} \langle j\rangle^2 |  v_{j,\lambda}(0,x)|^2\Big)^\frac12  \bigg\|_{L_x^6}^3
\lesssim  
\frac1{\lambda^2} \|\vec{v}_0\|_{ H^1_x h^1}^3.
\end{align*}
We can estimate $A_2$ similarly to $A_1$:
\begin{align*}
&\Bigg\|\sum_{j\in \mathbb{Z}} \sum_{ (j_1,j_2,j_3 ) \in \mathcal{NR}(j) } e^{-it (|j_1|^2 - |j_2|^2 + |j_3|^2)} e^{i\langle y,j\rangle }
\frac{i}{\Delta_{\mathbb{R}^2} + |j_1|^2- |j_2|^2 + |j_3|^2- |j|^2} P_{\le 2^{-10}}^x (v_{j_1,\lambda}\bar{v}_{j_2,\lambda} v_{j_3,\lambda})
(t ,x)\Bigg\|_{L_t^\infty L_x^2 H_{y}^{1}\cap L_{t,x}^4 H_y^1}\\
\lesssim &  \frac1{\lambda^2} \left\|\vec{v}\left(\frac{t }{\lambda^2}\right)\right\|_{L_t^\infty H^1_x h^1 }^3 \lesssim \frac1{\lambda^2} \|\vec{v}_0\|_{  H^1_x h^1 }^3.
\end{align*}
We now consider $A_3$. By H\"older's inequality, Leibnitz's rule and \eqref{eq3.4new}, we have
\begin{align}\label{eq3.1133}
& \Bigg \|\sum_{j\in \mathbb{Z}} \sum_{ (j_1,j_2,j_3) \in \mathcal{NR}(j) } e^{-i t (|j_1|^2 - |j_2|^2 + |j_3|^2)} e^{i\langle y,j\rangle}
\int_0^t  \frac{e^{i(t -\tau)(\Delta_{\mathbb{R}^2} + |j_1|^2 - |j_2|^2 + |j_3|^2 - |j|^2)}}{\Delta_{\mathbb{R}^2 } + |j_1|^2 - |j_2|^2 + |j_3|^2 - |j|^2}
\partial_\tau P_{\le 2^{-10}}^x (v_{j_1,\lambda} \bar{v}_{j_2,\lambda} v_{j_3,\lambda})  \,\mathrm{d}\tau \Bigg\|_{L_t^\infty L_x^2 H_{y}^{1}} \notag\\
& \ +  \Bigg \|\sum_{j\in \mathbb{Z}} \sum_{ (j_1,j_2,j_3) \in \mathcal{NR}(j) } e^{-i t (|j_1|^2 - |j_2|^2 + |j_3|^2)} e^{i\langle y,j\rangle}
\int_0^t  \frac{e^{i(t -\tau)(\Delta_{\mathbb{R}^2} + |j_1|^2 - |j_2|^2 + |j_3|^2 - |j|^2)}}{\Delta_{\mathbb{R}^2 } + |j_1|^2 - |j_2|^2 + |j_3|^2 - |j|^2}
\partial_\tau P_{\le 2^{-10}}^x (v_{j_1,\lambda} \bar{v}_{j_2,\lambda} v_{j_3,\lambda})  \,\mathrm{d}\tau \Bigg\|_{L_{t,x}^4 H_{y}^{1}} \notag\\
\lesssim &  \left\| \bigg\| \Big(\sum_{j\in \mathbb{Z}} \langle j\rangle^2 | \partial_t v_{j,\lambda}(t,x) |^2 \Big)^\frac12 \bigg\|_{L_x^6}
\bigg\|\Big(\sum_{j\in \mathbb{Z}} \langle j\rangle^2  |v_{j,\lambda}(t,x)|^2  \Big)^\frac12   \bigg\|_{L_x^6}^2\right \|_{L_t^1} \notag\\
\lesssim & 
 \lambda^{-2} \left\|   \bigg(\sum_{j\in \mathbb{Z}} \langle j\rangle^2\Big |\Big(  -\Delta_{\mathbb{R}^2} v_j + \sum_{(j_1,j_2,j_3)\in \mathcal{R}(j)} v_{j_1} \bar{v}_{j_2} v_{j_3}\Big)(t,x)\Big|^2 \bigg)^\frac12  \right\|_{L_t^3 L_x^6}
\left\|  \Big(\sum_{j\in \mathbb{Z}} \langle j\rangle^2 |v_j(t,x)|^2 \Big)^\frac12  \right\|_{L_t^3L_x^6}^2 \notag \\
=: &  \lambda^{-2} \cdot  I \cdot II .
\end{align}
By H\"older's inequality, Sobolev's inequality, together with the scattering of the cubic resonant Schr\"odinger system and persistence of regularity
in \cite{YZ}, 
 we have
\begin{align}\label{eq3.1333}
I \lesssim & \left\|    \Big(\sum_{j\in \mathbb{Z}} \langle j\rangle^2 |\Delta_{\mathbb{R}^2} v_j(t,x)|^2 \Big)^\frac12    \right\|_{L_t^3 L_x^6}
+  \left\|   \bigg(\sum_{j\in \mathbb{Z}} \langle j\rangle^2 \Big| \sum_{(j_1,j_2,j_3)\in \mathcal{R}(j)}(  v_{j_1} \bar{v}_{j_2} v_{j_3})(t,x) \Big|^2 \bigg)^\frac12 \right\|_{L_t^3 L_x^6} \notag\\
\lesssim & \|\vec{v}_0\|_{ H^2_x h^1}+
\left\|  \left(\sum_j \langle j \rangle^2 | v_j(t,x)|^2  \right)^\frac12  \right\|_{L_t^3L_x^{18}}
\left\| \Big(\sum_j \langle j \rangle^2 |v_j(t,x)|^2 \Big)^\frac12 \right\|_{L_t^\infty L_x^{18}}^2\notag\\
\lesssim & \|\vec{v}_0\|_{ H^2_x h^1} + \|\vec{v}_0\|_{H^2_x h^1 }^3.
\end{align}
By the scattering result in \cite{YZ}, we also have
\begin{equation}\label{eq3.1233}
II \lesssim  \|\phi\|_{L_x^2 H_y^1}^2. 
\end{equation}
So by \eqref{eq3.1133}, \eqref{eq3.1333} and \eqref{eq3.1233}, we have
\begin{align*}
&\  \Bigg \|\sum_{j\in \mathbb{Z}} \sum_{( j_1,j_2,j_3 ) \in \mathcal{NR}(j) } e^{-i t (|j_1|^2 - |j_2|^2 + |j_3|^2)} e^{i\langle y,j\rangle}
\int_0^t  \frac{e^{i(t -\tau)(\Delta_{\mathbb{R}^2} + |j_1|^2 - |j_2|^2 + |j_3|^2 - |j|^2)}}{\Delta_{\mathbb{R}^2 } + |j_1|^2 - |j_2|^2 + |j_3|^2 - |j|^2}
\partial_\tau P_{\le 2^{-10}}^x (v_{j_1,\lambda} \bar{v}_{j_2,\lambda} v_{j_3,\lambda})  \,\mathrm{d}\tau \Bigg\|_{L_t^\infty L_x^2 H_{ y}^{ 1} }\\
&\ +   \Bigg \|\sum_{j\in \mathbb{Z}} \sum_{( j_1,j_2,j_3 ) \in \mathcal{NR}(j) } e^{-i t (|j_1|^2 - |j_2|^2 + |j_3|^2)} e^{i\langle y,j\rangle}
\int_0^t  \frac{e^{i(t -\tau)(\Delta_{\mathbb{R}^2} + |j_1|^2 - |j_2|^2 + |j_3|^2 - |j|^2)}}{\Delta_{\mathbb{R}^2 } + |j_1|^2 - |j_2|^2 + |j_3|^2 - |j|^2}
\partial_\tau P_{\le 2^{-10}}^x (v_{j_1,\lambda} \bar{v}_{j_2,\lambda} v_{j_3,\lambda})  \,\mathrm{d}\tau \Bigg\|_{  L_{t,x}^4 H_y^1}\\
&  \lesssim_{\|\phi\|_{L_x^2 H_y^1} }   \lambda^{-2}\left(\|\vec{v}_0\|_{ H^2_x h^1} + \|\vec{v}_0\|_{H^2_x h^1 }^3\right).
\end{align*}
So $P_{\le 2^{-10}}^x e_\lambda$ is acceptable when $\lambda\ge \lambda_0$. Therefore, for $\lambda\ge \lambda_0$, $\int_0^t e^{i(t-\tau)\Delta_{\mathbb{R}^2
\times \mathbb{T}}} e_\lambda(\tau) \, \mathrm{d}\tau$ is small enough in $L_{t,x}^4 H_y^1 \cap L_t^\infty L_x^2 H_y^1(\mathbb{R}\times \mathbb{R}^2\times \mathbb{T})$.

We still need to verify the easier assumptions of Theorem \ref{le2.6}.

By Plancherel, \eqref{eq3.4new}, we have
\begin{align*}
 \|V_\lambda\|_{ L_t^\infty L_x^2  H_{ y}^{ 1}(\mathbb{R}\times  \mathbb{R}^2\times \mathbb{T} )}
\lesssim & \left( \sum_{j\in \mathbb{Z}} \langle j \rangle^2 \bigg(\|v_j(0,x)\|_{L_x^2}^2
+ \bigg\|\sum_{\substack{j_1-j_2+ j_3 =j,\\ |j_1|^2 - |j_2|^2 + |j_3|^2  = |j|^2}} v_{j_1} \overline{v_{j_2}} v_{j_3} \bigg\|_{L_{t,x}^\frac43 }^2\bigg) \right )^\frac12
\\
\lesssim  & \|\vec{v}_0\|_{ L^2_x h^1}
+  \left(\sum_j \langle j \rangle^2 \|v_j\|_{L_{t,x}^4}^2 \right)^\frac32
\lesssim 
  \|\vec{v}_0\|_{ L^2_x h^1}
+ \|\vec{v}_0\|_{ L^2_x h^1}^3,
\end{align*}
and
\begin{align*}
 \|V_\lambda\|_{L_t^4 L_x^4 H_y^1(\mathbb{R} \times \mathbb{R}^2 \times \mathbb{T})}
\lesssim & 
\left\|\sum_{j\in \mathbb{Z}} e^{i\langle y,j\rangle} v_j(0,x)\right\|_{L_x^2 H_y^1}
+ \left\| \sum_{j\in \mathbb{Z}} e^{-it|j|^2} e^{i\langle y,j\rangle} (i\partial_t + \Delta_{\mathbb{R}^2})v_j(t,x)\right\|_{L_t^\frac43 L_x^\frac43 H_y^1}\\
\sim & \left(\sum\limits_{j\in \mathbb{Z}} \langle j\rangle^2 \|v_j(0,x)\|_{L_x^2}^2\right)^\frac12
 +  \left\| \left(\sum_{j\in \mathbb{Z}} \langle j\rangle^2 |(i\partial_t + \Delta_{\mathbb{R}^2})v_j(t,x)|^2 \right)^\frac12\right\|_{L_{t,x}^\frac43}\\
\lesssim &
\|\vec{v}_0\|_{ L^2_x h^1} +
\left\| \Big(\sum_{j} \langle j \rangle^2 |v_{j}(t,x)|^2\Big)^\frac12 \right\|_{L_{t,x}^4}^3
\lesssim \|\vec{v}_0\|_{ L^2_x h^1} + \|\vec{v}_0\|_{ L^2_x h^1}^3.
\end{align*}
Moreover, by Plancherel, \eqref{eq3.935}, we have
\begin{align*}
 \|U_\lambda(0) - V_\lambda(0)\|_{L_x^2 H_{y}^1}
=
 \|\vec{\phi} - \vec{v}_{0}\|_{L_x^2 h^1}
\lesssim \epsilon_1.
\end{align*}
Applying Theorem \ref{le2.6}, we conclude that for $\lambda$ (depending on $\vec{v}_0$) large enough, the solution $U_\lambda$ of \eqref{eq3.833}
exists globally 
 and
\begin{align*}
\|U_\lambda- V_\lambda\|_{L_t^\infty L_x^2 H_{y}^1 \cap L_{t,x}^4 H_y^1(\mathbb{R} \times \mathbb{R}^2 \times \mathbb{T})} \lesssim \epsilon_1,
\end{align*}
which ends the proof.
\end{proof}
After spatial translation, time translation, and Galilean transformation, we have
\begin{proposition}\label{pr5.9}
For any $\phi\in L_x^2 H_y^{ 1}(\mathbb{R}^2 \times \mathbb{T})$,
$(\lambda_n, t_n, x_n,\xi_n )_{n\ge 1} \subseteq (0,\infty) \times \mathbb{R} \times \mathbb{R}^2 \times \mathbb{R}^2$, $\lambda_n\to \infty$, as $n\to \infty$, $|\xi_n| \lesssim 1$, and let
\begin{align*}
u_n(0,x,y) = \frac1{\lambda_n}  e^{ix \xi_n} (e^{it_n \Delta_{\mathbb{R}^2  }} \phi)\left(\frac{x-x_n}{\lambda_n},y\right).
\end{align*}

(i) For $n$ large enough, there is a solution $u_n \in C_t^0 L_x^2 H_{y}^{ 1} \cap L_{t,x}^4 H_y^1(\mathbb{R} \times \mathbb{R}^2\times \mathbb{T})$
of \eqref{eq1.1}, satisfying
$\|U_k\|_{L_t^\infty L_x^2 H_{y}^{ 1}(\mathbb{R}\times \mathbb{R}^2 \times \mathbb{T})} \lesssim_{\|\psi\|_{L_x^2 H_y^1}} 1$.

(ii) There exists a solution $\vec{v} \in C_t^0 L^2_x h^1 (\mathbb{R} \times \mathbb{R}^2\times \mathbb{Z} )$ of the resonant Schr\"odinger system
 such that: 
for any $\epsilon > 0$, 
it holds that
\begin{align*}
\|u_n(t) - W_n(t)\|_{L_t^\infty L_{x }^2 H_y^1 \cap L_{t,x}^4 H_y^1(\mathbb{R}\times \mathbb{R}^2 \times \mathbb{T}) } \le \epsilon,\\
\|u_n\|_{L_t^\infty L_{x }^2 H_y^1 \cap L_{t,x}^4 H_y^1(\mathbb{R}\times \mathbb{R}^2 \times \mathbb{T}) } \lesssim 1.
\end{align*}
for $n$ large enough, where
\begin{align*}
W_n(t, x,y) =   e^{-it_n  \Delta_{\mathbb{T}}}¡¡V_{\lambda_n}(t-t_n, x-x_n, y).
\end{align*}
\end{proposition}
\begin{proof}
Without loss of generality, we may assume that $x_n = 0$.
Using a Galilean transform and the fact that $\xi_n$ is bounded, we may assume that $\xi_n = 0$ for
all $n$. Similar to the argument of \cite{HP}, 
we obtain the result.
\end{proof}

\section{Existence of an almost-periodic solution} \label{se4}
In this section, we will show the existence in $L_x^2 H^{1}_{y}(\mathbb{R}^2 \times \mathbb{T})$ of an almost-periodic solution
by the profile decomposition.
By Theorem \ref{th2.946}, to prove the scattering of the solution of \eqref{eq1.1},
we only need to show the finiteness of the space-time norm $\|\cdot \|_{L_{t,x}^4 H_y^{1-\epsilon_0}}$ of the solution $u$ of \eqref{eq1.1}.
Define
\begin{equation*}
\Lambda(L) = \sup  \, \|  u\|_{L_t^4 L_x^{4} H_y^{1-\epsilon_0}  (\mathbb{R} \times \mathbb{R}^2\times \mathbb{T})},
\end{equation*}
where the supremum is taken over all global solutions $u \in C_t^0 H_{x,y}^{ 1}$ of \eqref{eq1.1} obeying $\sup\limits_{t\in \mathbb{R} } \|u(t)\|_{L_x^2H_y^{ 1}}^2 \le L$.

By the local wellposedness theory, $\Lambda(L)< \infty$ for $L$ sufficiently small.
In addition, define $L_{max} = \sup \{ L: \Lambda(L) < \infty\}$.
Our goal is to prove $L_{max}  = \infty$. Suppose to the contradiction $L_{max} < \infty$, we will show a Palais-Smale type theorem.
\begin{proposition}[Palais-Smale condition modulo symmetries in $L_x^2 H_y^{ 1}(\mathbb{R}^2 \times \mathbb{T})$] \label{pr7.1}
Assume that $L_{max} < \infty$. Let $\{t_n\}_n$ be arbitrary sequence of real numbers and $\{u_n\}$ be a sequence of solutions to
\eqref{eq1.1} satisfying
\begin{align}
 & \sup_{t\in (-\infty,\infty)} \|u_n(t)\|_{L_x^2 H_y^{ 1}}^2  \to L_{max},\\
 & \|u_n\|_{L_{t,x}^4 H_y^{1-\epsilon_0} ((-\infty,t_n) \times \mathbb{R}^2\times \mathbb{T})} \to \infty,
\|u_n\|_{ L_{t,x}^4 H_y^{1-\epsilon_0}((t_n,\infty) \times \mathbb{R}^2\times \mathbb{T})} \to \infty, \text{ as } n \to \infty, \label{eq3.13n0}
\end{align}
and such that $u_n\in C_t^0 H_{x,y}^{1}((-\infty,\infty)\times \mathbb{R}^2 \times \mathbb{T})$. Then, after passing to a subsequence, there exists a sequence
$x_n \in \mathbb{R}^2$ and $w\in H^{ 1}(\mathbb{R}^2\times \mathbb{T})$ such that
\begin{equation*}
  u_n  (x +  x_n, y, t_n)    \to w(x,y) \text{ in } L_x^2 H_y^{ 1}(\mathbb{R}^2 \times \mathbb{T}), \text{ as } n\to \infty.
\end{equation*}
\end{proposition}
\begin{proof}
By replacing $u_n(t)$ with $u_n(t+t_n)$, we may assume $t_n = 0$. Applying Proposition \ref{sse3.1} to $\{u_n(0)\}_n$, after passing to a subsequence, we have
\begin{align*}
u_n(0,x,y) &  = \sum\limits_{j=1}^J  \frac1{\lambda_n^j}   e^{ix\xi_n^j}   \left(e^{it_n^j \Delta_{\mathbb{R}^2  }}   \phi^j\right)
\left(\frac{x-x_n^j}{\lambda_n^j},y\right)   + w_n^J(x,y).
\end{align*}
The remainder has asymptotically trivial linear evolution
\begin{equation}\label{eq6.1}
\limsup\limits_{n\to \infty} \|e^{it\Delta_{\mathbb{R}^2 \times \mathbb{T}  }}w_n^J \|_{ L_t^4 L_x^4 H_y^{1-\epsilon_0}} \to 0, \text{ as } J \to \infty,
\end{equation}
and we also have asymptotic decoupling of the $L_x^2 H_y^1$ norm:
\begin{equation}\label{eq6.2}
\lim\limits_{n\to \infty} \left(\|u_n(0)\|_{L_x^2 H_y^{1}}^2 - \sum\limits_{j=1}^J \|\phi^j\|_{L_x^2 H_y^{1}}^2 - \|w_n^J\|_{L_x^2H_y^{1}}^2 \right) = 0,\  \forall\, J.
\end{equation}
There are two possibilities:

{\bf Case 1.}  $\sup\limits_{j} \limsup\limits_{n\to \infty} \|\phi^j\|_{L_x^2 H_y^{ 1}}^2 = L_{max}$.
Combining \eqref{eq6.2} with the fact that $\phi^j$ are nontrivial in $L_x^2 H_y^{1}$, we deduce that $u_n(0,x,y) =  \frac1{\lambda_n} e^{ix\xi_n}    \left(e^{it_n\Delta_{  \mathbb{R}^2 }}   \phi\right)\left(\frac{x-x_n}{\lambda_n},y\right)   + w_n(x,y) $, $\lim\limits_{n\to \infty} \|w_n\|_{L_x^2 H_y^{ 1}} = 0$.
We will show that $\lambda_n \equiv 1$, otherwise $\lambda_n \to  \infty$.

Proposition \ref{pr5.9} 
 implies that for all large $n$, there exists a unique solution $u_n$ on $\mathbb{R}$ with $u_n(0,x,y) = \frac1{\lambda_n}  e^{ix\xi_n}      (e^{it_n \Delta_{\mathbb{R}^2 }}   \phi)(\frac{x-x_n}{\lambda_n}, y)$
and
\begin{align*}
\limsup\limits_{n\to \infty} \|u_n\|_{ L_t^4 L_x^4 H_y^{1-\epsilon_0} (\mathbb{R} \times \mathbb{R}^2 \times \mathbb{T})} \le C(L_{max}),
\end{align*}
which is a contradiction with \eqref{eq3.13n0}.

Therefore, $\lambda_n \equiv 1$, and $u_n(0,x,y) =   e^{ix\xi_n}     \left(e^{it_n \Delta_{\mathbb{R}^2 } }  \phi\right)(x-x_n,y)   + w_n(x,y)$.
If $t_n \equiv 0$, by the fact $\xi_n $ is bounded, this is precisely the conclusion.
If $t_n \to -\infty$, by the Galilean transform
\begin{align*}
e^{it_0 \Delta_{\mathbb{R}^2}} e^{ix\xi_0} \phi(x) = e^{-it_0 |\xi_0|^2} e^{ix\xi_0} (e^{it_0 \Delta_{\mathbb{R}^2}} \phi)(x-2t_0 \xi_0),
\end{align*}
we observe
\begin{align*}
&\left \|e^{it\Delta_{\mathbb{R}^2\times \mathbb{T}}}   \left( e^{ix\xi_n}   ( e^{it_n\Delta_{\mathbb{R}^2  }}  \phi)(x-x_n,y) \right)  \right\|_{ L_{t,x}^4 H_y^{1-\epsilon_0} ((-\infty,0) \times \mathbb{R}^2\times \mathbb{T})}\\
= & \left\|e^{-it|\xi_n|^2} e^{ix\xi_n} (e^{it\Delta_{\mathbb{R}^2}} (e^{it_n\Delta_{\mathbb{R}^2}} \phi)(\cdot -x_n,y))(x-2t\xi_n) \right\|_{ L_{t,x}^4 H_y^{1-\epsilon_0} ((-\infty,0) \times \mathbb{R}^2\times \mathbb{T})}\\
=   & \ \|e^{i(t+t_n) \Delta_{\mathbb{R}^2 }} \phi\|_{L_{t,x}^4 H_y^{1-\epsilon_0}((-\infty,0)\times \mathbb{R}^2 \times \mathbb{T})}
=   \ \|e^{it\Delta_{\mathbb{R}^2  }} \phi\|_{ L_{t,x}^4 H_y^{1-\epsilon_0}((-\infty,t_n)\times \mathbb{R}^2 \times \mathbb{T})}
\to 0, \text{ as } n\to \infty.
 \end{align*}
Using Theorem \ref{th2.3}, we see that, for $n$ large enough,
\begin{align*}
\|u_n\|_{ L_{t,x}^4 H_y^{1-\epsilon_0}((-\infty, 0) \times \mathbb{R}^2 \times \mathbb{T})} \le 2 \delta_0 < \infty,
\end{align*}
which contradicts \eqref{eq3.13n0}.
The case $t_n \to \infty$ is similar.

{ \bf Case 2.}  $\sup\limits_{j} \limsup\limits_{n\to \infty} \|\phi^j\|_{L_x^2 H_y^{ 1}}^2 \le L_{max} - 2\delta$ for some $\delta > 0$.

By the definition of $L_{max}$, there exist global solution $v_n^j$ to
\begin{equation*}
\begin{cases}
i\partial_t v_n^j + \Delta_{\mathbb{R}^2\times \mathbb{T}}  v_n^j = |v_n^j|^2 v_n^j,\\
v_n^j(0,x,y) =  \frac1{\lambda_n^j}   e^{ix\xi_n^j}   \left(e^{it_n^j \Delta_{\mathbb{R}^2  }}   \phi^j\right)
\left(\frac{x-x_n^j}{\lambda_n^j},y\right),
\end{cases}
\end{equation*}
satisfying
\begin{align}\label{eq6.3}
\| v_n^j\|_{L_{t,x}^4 H_y^{1-\epsilon_0} } \lesssim_{L_{max},\delta} \|\phi^j\|_{L_x^2H_y^{1}}.
\end{align}
Put
\begin{align*}
u_n^J = \sum\limits_{j=1}^J v_n^j + e^{it\Delta_{\mathbb{R}^2 \times \mathbb{T}}} w_n^J.
\end{align*}
Then we have $u_n^J(0) = u_n(0)$. We claim that for sufficiently large $J$ and $n$, $u_n^J$ is an approximate solution to $u_n$ in the sense of the Theorem \ref{le2.6}. Then we have the finiteness of $\|u_n\|_{L_{t,x}^4 H_y^{1-\epsilon_0}(\mathbb{R} \times \mathbb{R}^2 \times \mathbb{T} ) }$, which contradicts with \eqref{eq3.13n0}.

To verify the claim, we only need to check that $u_n^J$ satisfies the following properties:

$(i)$ $\limsup\limits_{n\to \infty} \|u_n^J\|_{ L_t^4 L_x^4 H_y^{1-\epsilon_0} } \lesssim_{L_{max},\delta} 1$, uniformly in $J$;

$(ii)$ $  \limsup\limits_{n\to \infty} \|e_n^J\|_{ L_{t,x}^\frac43 H_y^{1-\epsilon_0} } \to 0$, as $ J\to J^* $, where $e_n^J = (i\partial_t + \Delta_{\mathbb{R}^2\times \mathbb{T}})u_n^J - |u_n^J|^2 u_n^J$.

The verification of $(i)$ relies on the asymptotic decoupling of the nonlinear profiles $v_n^j$, which we record in the following two lemmas.
\begin{lemma}[Orthogonality]
Suppose that two frames $\Gamma^j = (t_n^j, x_n^j, \xi_n^j, \lambda_n^j)$, $\Gamma^k = (t_n^k, x_n^k, \xi_n^k, \lambda_n^k)$ are orthogonal,
then for $\psi^j,\, \psi^k \in C_0^\infty(\mathbb{R}\times \mathbb{R}^2 \times \mathbb{T})$,
\begin{align*}
\left\| \frac1{\lambda_n^j} e^{ix\xi_n^j} \left(  \langle \nabla_y \rangle^{1-\epsilon_0} \psi^j\right) \left(\frac{t-t_n^j}{(\lambda_n^j)^2} ,\frac{x-x_n^j}{\lambda_n^j}, y\right) \cdot \frac1{\lambda_n^k} e^{ix\xi_n^k} \left(\langle \nabla_y  \rangle^{1-\epsilon_0} \psi^k\right) \left(\frac{t-t_n^k}{(\lambda_n^k)^2}, \frac{x-x_n^k}{\lambda_n^k},y\right)\right\|_{L_{t,x,y}^2  } &  \to 0,
\text{  as } n\to \infty.
\end{align*}
\end{lemma}
Similarly to the proof in \cite{HP} to deal with the quintic nonlinear Schr\"odinger equation on $\mathbb{R}\times \mathbb{T}^2$,
we can obtain the following lemma from Proposition \ref{pr5.9}. We also refer to \cite{CMZ} for similar argument.
\begin{lemma}[Decoupling of nonlinear profiles]\label{le6.3}
Let $v_n^j$ be the nonlinear solutions defined above, then for $j\ne k$,
\begin{align}\label{eq7.29}
&  \left\|\langle \nabla_y  \rangle^{1-\epsilon_0}  v_n^j \cdot  \langle \nabla_y
 \rangle^{1-\epsilon_0} v_n^k   \right\|_{L_{t,x,y}^2} \to 0,\ 
\text{ and }   
\left \|v_n^k \cdot \langle \nabla_y  \rangle^{1-\epsilon_0} v_n^j\right\|_{L_{t,x,y}^2} \to 0,
 \text{ as } n\to \infty.
\end{align}
\end{lemma}
Let us verify claim $(i)$ above. We see
\begin{align*}
\|u_n^J \|_{ L_{t,x}^4 H_y^{1-\epsilon_0} }     & \le \left\|\sum\limits_{j=1}^J v_n^j \right\|_{  L_{t,x}^4 H_y^{1-\epsilon_0} } + \left \|e^{it\Delta_{\mathbb{R}^2 \times \mathbb{T}}} w_n^J \right \|_{ L_{t,x}^4 H_y^{1-\epsilon_0}}.
\end{align*}
By \eqref{eq6.3} and Lemma \ref{le6.3}, we have
\begin{align*}
   \left\|\sum\limits_{j=1}^J v_n^j \right\|_{  L_{t,x}^4 H_y^{1-\epsilon_0}}^4
 \lesssim  & \left(\sum\limits_{j=1}^J\left\|    v_n^j     \right\|_{L_{t,x}^4 H_y^{1-\epsilon_0}}^2 +
 \sum\limits_{j\ne k} \left\|    \langle \nabla_y  \rangle^{1-\epsilon_0} v_n^j \cdot  \langle \nabla_y
 \rangle^{1-\epsilon_0} v_n^k     \right\|_{L_{t,x,y}^2}\right)^2\\
 \lesssim &  \left(\sum\limits_{j=1}^J \|\phi^j\|_{L_x^2H_y^{ 1}}^2 + o_J(1)\right)^2.
\end{align*}
Since the $L_x^2 H_y^1$ norm decoupling implies
\begin{align*}
\limsup\limits_{n\to \infty}  \sum\limits_{j=1}^J \|\phi^j\|_{L_x^2H_y^{ 1}}^2 \le L_{max},
\end{align*}
together with \eqref{eq5.3}, we obtain
\begin{align}\label{eq6.5}
\lim\limits_{J\to J^*} \limsup\limits_{n\to \infty} \|  u_n^J\|_{L_t^4 L_x^{ 4} H_y^{1-\epsilon_0}} \lesssim_{L_{max}, \delta} 1.
\end{align}
It remains to check property $(ii)$ above, by the definition of $u_n^J$, we decompose
\begin{align*}
e_n^J & = (i\partial_t + \Delta_{\mathbb{R}^2\times \mathbb{T}} )u_n^J - |u_n^J|^2 u_n^J\\
      & = \sum\limits_{j=1}^J |v_n^j|^2 v_n^j - \left|\sum\limits_{j=1}^J v_n^j\right|^2 \sum\limits_{j=1}^J v_n^j
      + |u_n^J- e^{it\Delta_{\mathbb{R}^2\times \mathbb{T}}} w_n^J|^2 (u_n^J- e^{it\Delta_{\mathbb{R}^2\times \mathbb{T}}} w_n^J) - |u_n^J|^2 u_n^J.
\end{align*}
First consider
\begin{align*}
\sum\limits_{j=1}^J |v_n^j|^2 v_n^j - \left|\sum\limits_{j=1}^J v_n^j\right|^2 \sum\limits_{j=1}^J v_n^j.
\end{align*}
Thus by the fractional chain rule, Minkowski, H\"older, Sobolev, \eqref{eq7.29} and \eqref{eq6.3},
\begin{align}\label{eq4.832}
& \left\| \sum\limits_{j=1}^J |v_n^j|^2 v_n^j - \left|\sum\limits_{j=1}^J v_n^j\right|^2 \sum\limits_{j=1}^J v_n^j\right\|_{
L_{t,x}^4 H_y^{1-\epsilon_0}}\\
\sum\limits_{j\ne k} \|v_n^k \langle \nabla_y \rangle^{1-\epsilon_0} v_n^j \|_{L_{t,x,y}^2} \|v_n^k\|_{L_{t,x}^4 H_y^{1-\epsilon_0}}
+ \sum\limits_{j\ne k} \|v_n^k \langle \nabla_y  \rangle^{1-\epsilon_0} v_n^j \|_{L_{t,x,y}^2} \|v_n^j\|_{L_{t,x}^4 H_y^{1-\epsilon_0} }
\lesssim o_J(1), \text{ as } n\to \infty.\notag
\end{align}
We now estimate $\left|u_n^J  - e^{it\Delta_{\mathbb{R}^2 \times \mathbb{T}}} w_n^J\right|^2(u_n^J  - e^{it\Delta_{\mathbb{R}^2 \times \mathbb{T}}} w_n^J) - |u_n^J|^2 u_n^J$.
By the fractional chain rule, H\"older, Sobolev, we have
\begin{align*}
& \left\|\left|u_n^J- e^{it\Delta_{\mathbb{R}^2 \times \mathbb{T}}} w_n^J\right|^2\left(u_n^J  - e^{it\Delta_{\mathbb{R}^2 \times \mathbb{T}}} w_n^J\right) - |u_n^J|^2 u_n^J\right \|_{  L_t^\frac43 L_x^\frac43 H_y^{1-\epsilon_0} }\\
\lesssim & \left\|  \left(\|u_n^J\|_{H_y^{1-\epsilon_0}}^2 + \|e^{it\Delta_{\mathbb{R}^2\times \mathbb{T}}} 
w_n^J \|_{H_y^{1-\epsilon_0}}^2 \right) \|e^{it\Delta_{\mathbb{R}^2\times \mathbb{T}}} w_n^J\|_{H_y^{1-\epsilon_0}} \right\|_{L_{t,x}^\frac43} \\
 \lesssim & \left(\|u_n^J\|_{L_{t,x}^4 H_y^{1-\epsilon_0}}^2 + \|e^{it\Delta_{\mathbb{R}^2 \times \mathbb{T}}} w_n^J \|_{L_{t,x}^4 H_y^{1-\epsilon_0}}^2 \right)
\|e^{it\Delta_{\mathbb{R}^2\times \mathbb{T}}} w_n^J\|_{L_{t,x}^4 H_y^{1-\epsilon_0}  }.
 \end{align*}
Using \eqref{eq6.5}, and the decay property \eqref{eq6.1}, we get
\begin{align*}
\limsup\limits_{n\to \infty} \left\||u_n^J- e^{it\Delta_{\mathbb{R}^2 \times \mathbb{T}}} w_n^J|^2( u_n^J- e^{it\Delta_{\mathbb{R}^2 \times \mathbb{T}}} w_n^J  ) - |u_n^J|^2 u_n^J \right\|_{  L_{t,x}^4 H_y^{1-\epsilon_0}} \to 0, \text{ as }  J\to J^*.
\end{align*}
\end{proof}
Arguing as in \cite{CMZ}, the proof of Proposition \ref{pr7.1} implies the following result:
\begin{theorem}[Existence of the almost-periodic solution] \label{co4.727} 
Assume that $L_{max} < \infty$, then there exists $u_c\in C_{t }^0 H_{x,y}^1(\mathbb{R}\times \mathbb{R}^2 \times \mathbb{T})$ solving \eqref{eq1.1} satisfying
\begin{align}
\sup\limits_{t\in \mathbb{R}} \|u_c(t)\|_{L_x^2H_y^1(\mathbb{R}^2 \times \mathbb{T})}^2 = L_{max},\
\|u_c\|_{L_{t,x}^4 H_y^{1-\epsilon_0}(\mathbb{R}\times \mathbb{R}^2 \times \mathbb{T})} = \infty.
\end{align}
Furthermore, $u_c$ is almost periodic in the sense that
$\forall \eta > 0$, there is $C(\eta) > 0$ such that
\begin{align}\label{eq4.1735}
\int_{|x+x(t)| \ge C(\eta)} \|   u_c(t,x,y)\|_{H_y^1(\mathbb{T})}^2 \mathrm{d}x   < \eta
\end{align}
for all $t\in \mathbb{R}$, where $ {x}: \mathbb{R}\to \mathbb{R}^2$ is a Lipschitz function with $\sup\limits_{t\in \mathbb{R}} | {x'}(t)| \lesssim 1$.
\end{theorem}

\section{Rigidity theorem}\label{se5}
In this section, we will exclude the almost-periodic solution in Theorem \ref{co4.727} by using the interaction Morawetz action with the weight function taken appropriately.
\begin{proposition}[Non-existence of the almost-periodic solution ]\label{pr7.3}
The almost-periodic solution $u_c$ in Theorem \ref{co4.727} does not exist. 
\end{proposition}
\begin{proof}
Define the interaction Morawetz action
\begin{align*}
M(t) = \int_{\mathbb{R}^2 \times \mathbb{T}} \int_{\mathbb{R}^2\times \mathbb{T}} \Im(\bar{u}_c(t,x,y) \nabla_x u_c(t,x,y) ) \nabla_x a(x-\tilde{x}) |u_c(t,\tilde{x},\tilde{y})|^2 \,\mathrm{d}x\mathrm{d}y \mathrm{d}\tilde{x} \mathrm{d}\tilde{y},
\end{align*}
where $a$ is a radial function defined on $\mathbb{R}^2$ with
\begin{align*}
a(r) =
\begin{cases}
\frac{r^2}{2r_0}(1+ \frac12 \log\frac{r_0}r),  \ & r < r_0,\\
r- \frac{r_0}2,                                 \ &  r \ge r_0,
\end{cases}
\end{align*}
$r_0$ is some positive constant.
We have $\Delta a \ge 0$ and $\|\frac3{4r_0} + \frac1{2r_0} \log\frac{r_0}r\|_{L^2(r \le r_0)} \lesssim 1$, and for $r > r_0$, $\Delta a   = \frac1r$.

Letting $r_0 \to 0$, we have $\forall\, T_0 > 0  $,
\begin{align}\label{eq5.032}
\int_{-T_0}^{T_0} \int_{\mathbb{R}^2 \times \mathbb{T}} \big||\nabla_x |^\frac12 (|u_c(t,x,y)|^2)\big|^2 \,\mathrm{d}x \mathrm{d}y \mathrm{d}t
\lesssim \int_{-T_0 }^{T_0}  M'(t) \,\mathrm{d}t \lesssim \sup_{t\in [-T_0,T_0] } |M(t)|.
\end{align}

By H\"older, Minkowski, Sobolev, and the fact $\mathbb{T}$ is compact, we have
\begin{align}\label{eq5.132}
\int_{|x + x(t)| \le C\big(\frac{m_0 }{100}\big)} \|u_c(t,x,y)\|_{L_y^2}^2 \,\mathrm{d}x
&  \lesssim C\left(\frac{m_0 }{100}\right)^\frac32 \|u_c(t,x,y)\|_{L_x^8 L_y^2}^2   \\
&  \lesssim C\left(\frac{m_0 }{100}\right)^\frac32 \|u_c(t,x,y)\|_{L_y^2 L_x^8  }^2  \notag\\
& \lesssim C\left(\frac{m_0 }{100}\right)^\frac32 \left\||\nabla_x|^\frac12 (|u_c(t,x,y)|^2)\right\|_{  L_y^1 L_x^2 }\lesssim C\left(\frac{m_0 }{100}\right)^\frac32 \left\||\nabla_x|^\frac12(|u_c(t,x,y)|^2)\right\|_{L_{x,y}^2}, \notag
\end{align}
where $m_0  = \mathcal{M}(u_c)$.

By \eqref{eq4.1735} and conservation of mass, we have
\begin{align}\label{eq5.232}
\frac{m_0 }2 \le \int_{|x + x(t)|\le C\big(\frac{m_0 }{100}\big)} \|  u_c(t,x,y)\|_{L_y^2}^2 \,\mathrm{d}x.
\end{align}
By H\"older, the fact $|\nabla a|$ is bounded, and $u_c \in C_t^0 H_{x,y}^1(\mathbb{R}\times \mathbb{R}^2 \times \mathbb{T})$,
\begin{align}\label{eq5.432}
|M(t)|  &  \lesssim \|u_c\|_{L_t^\infty L_{x,y}^2}^3 \|\nabla_x u_c \|_{L_t^\infty L_{x,y}^2} \lesssim 1.
\end{align}
Therefore, $\forall \ T_0 > 0$, by \eqref{eq5.232}, \eqref{eq5.132}, \eqref{eq5.032} and \eqref{eq5.432},
\begin{align*}
m_0^2 T_0 =  \int_0^{T_0}  m^2_0 \,\mathrm{d}t & \lesssim \int_{-T_0}^{T_0} \left(\int_{|x + x(t)|\le C\big(\frac{m_0 }{100}\big)} \|u_c(t,x,y)\|_{L_y^2}^2 \,\mathrm{d}x \right)^2 \,\mathrm{d}t\\
                             & \lesssim C\Big(\frac{m_0 }{100}\Big)^3 \int_{-T_0}^{T_0} \left\||\nabla_x|^\frac12 (|u_c(t,x,y)|^2) \right\|_{L_{x,y}^2}^2 \,\mathrm{d}t \\
                             &  \sim   C\Big(\frac{m_0 }{100}\Big) \left\||\nabla_x|^\frac12(|u_c(t,x,y)|^2) \right\|_{L_{t,x,y}^2}^2
                           \lesssim  C\Big(\frac{m_0 }{100}\Big) \sup_{t\in [-T_0,T_0]} |M(t)| \lesssim C\Big(\frac{m_0 }{100}\Big).
\end{align*}
 Let $T_0 \to \infty$, we obtain a contradiction unless $u_c \equiv 0$, which is impossible due to
$\|u_c\|_{L_{t,x}^4 H_y^{1-\epsilon_0}(\mathbb{R}\times \mathbb{R}^2 \times \mathbb{T} )}   = \infty$.
\end{proof}

\end{document}